\documentclass[11pt, reqno]{amsart}
\usepackage{amsmath,amsopn,amssymb,amsthm,multicol}
\usepackage[cp1251]{inputenc}
\usepackage[english]{babel}
\usepackage[breaklinks=true,colorlinks=true,linkcolor=blue,citecolor=blue,urlcolor=blue]{hyperref}
\usepackage[dvipdf]{graphicx}
\usepackage{mathrsfs}

\renewenvironment{proof}[1][Proof]{\textbf{#1.} }
{\ \rule{0.5em}{0.5em}}

\renewcommand{\arraystretch}{1.5}

\newtheorem{theorem}{Theorem}

\newtheorem{lemma}{Lemma}
\newtheorem{remark}{Remark}
\newtheorem{corollary}{Corollary}

\newtheorem{example}{Example}

\begin{document}

\title[Ricci curvature and normalized Ricci flow ... \dots]
{Ricci curvature and normalized Ricci flow on generalized Wallach spaces}

\author{N.\,A.~Abiev}
\address{N.\,A.~Abiev \newline
Institute of Mathematics NAS KR, Bishkek, prospect Chui, 265a, 720071, Kyrgyzstan}
\email{abievn@mail.ru}

\begin{abstract}
We proved that  the normalized
Ricci flow does not preserve the positivity of Ricci curvature of  Riemannian metrics
on every generalized Wallach space with $a_1+a_2+a_3\le 1/2$,
in particular on the spaces $\operatorname{SU}(k+l+m)/\operatorname{SU}(k)\times \operatorname{SU}(l)
\times \operatorname{SU}(m)$
and
$\operatorname{Sp}(k+l+m)/\operatorname{Sp}(k)\times \operatorname{Sp}(l)
\times \operatorname{Sp}(m)$ independently on $k,l$ and $m$.
The positivity of Ricci curvature is preserved  for  all original metrics with
$\operatorname{Ric}>0$ on generalized Wallach spaces  $a_1+a_2+a_3> 1/2$
if the conditions $4\left(a_j+a_k\right)^2\ge (1-2a_i)(1+2a_i)^{-1}$ hold
for all  $\{i,j,k\}=\{1,2,3\}$.
We also established that the spaces
$\operatorname{SO}(k+l+m)/\operatorname{SO}(k)\times \operatorname{SO}(l)\times \operatorname{SO}(m)$ satisfy the above conditions for $\max\{k,l,m\}\le 11$, moreover,
additional conditions were found to keep $\operatorname{Ric}>0$
in cases when $\max\{k,l,m\}\le 11$ is violated.
Similar questions have also been studied for all other generalized Wallach spaces given in the classification of Yuri\u\i\ Nikonorov.

\vspace{2mm} \noindent Key words and phrases:
generalized Wallach space,
invariant Riemannian metric, normalized Ricci flow, sectional curvature,
Ricci curvature, dynamical system.
\vspace{2mm}

\end{abstract}

\maketitle

\section*{Introduction}\label{vvedenie}

In this paper we continue our studies~\cite{Ab7, AN}  related to the evolution of invariant Riemannian metrics on a class of homogeneous Riemannian spaces called generalized Wallach spaces (GWS) under the normalized Ricci flow (NRF)
\begin{equation}\label{ricciflow}
\dfrac {\partial}{\partial t} \bold{g}(t) = -2 \operatorname{Ric}_{\bold{g}}+ 2{\bold{g}(t)}\frac{S_{\bold{g}}}{n},
\end{equation}
introduced  in~\cite{Ham} by Hamilton, where $\operatorname{Ric}_{\bold{g}}$ is the Ricci tensor and $S_{\bold{g}}$ is the scalar curvature of an one-parametric family of Riemannian metrics~$\bold{g}(t)$ on a given Riemannian manifold of dimension~$n$.

According to~\cite{Nikonorov1}
a generalized Wallach space is
a~homogeneous almost effective compact space $G/H$ with  a  compact semisimple
connected Lie group~$G$ and its closed subgroup~$H$
(with corresponding Lie algebras~$\mathfrak{g}$ and~$\mathfrak{h}$)
such that
the isotropy representation of~$G/H$  admits a decomposition into a direct sum
$\mathfrak{p}=\mathfrak{p}_1\oplus \mathfrak{p}_2\oplus \mathfrak{p}_3$ of three
$\operatorname{Ad}(H)$-invariant irreducible modules
$\mathfrak{p}_1$, $\mathfrak{p}_2$ and $\mathfrak{p}_3$
satisfying
$[\mathfrak{p}_i,\mathfrak{p}_i]\subset \mathfrak{h}$ for each $i\in\{1,2,3\}$,
where~$\mathfrak{p}$ is an orthogonal complement of~$\mathfrak{h}$ in~$\mathfrak{g}$ with respect to a~bi-invariant inner product $\langle\boldsymbol{\cdot}
\,,\boldsymbol{\cdot}\rangle=-B(\boldsymbol{\cdot}\,,\boldsymbol{\cdot})$
on~$\mathfrak{g}$ defined by the Killing form
$B(\boldsymbol{\cdot}\,,\boldsymbol{\cdot})$ of~$\mathfrak{g}$.
For any fixed bi-invariant inner product
$\langle\cdot, \cdot\rangle$ on the Lie algebra $\mathfrak{g}$ of the Lie group~$G$,
any $G$-invariant Riemannian metric $\bold{g}$ on $G/H$ can be determined by an $\operatorname{Ad} (H)$-invariant inner product
\begin{equation}\label{metric}
(\cdot, \cdot)=\left.x_1\langle\cdot, \cdot\rangle\right|_{\mathfrak{p}_1}+
\left.x_2\langle\cdot, \cdot\rangle\right|_{\mathfrak{p}_2}+
\left.x_3\langle\cdot, \cdot\rangle\right|_{\mathfrak{p}_3},
\end{equation}
where $x_1,x_2,x_3$ are positive real numbers (see~\cite{Lomshakov2, Nikonorov2, Nikonorov4, Nikonorov1} and references therein for details).
We complete a brief description of these spaces recalling that one
of their important characteristics is that
every generalized Wallach space can be described by a triple of real numbers
$a_i:=A/d_i\in (0,1/2]$, $i=1,2,3$, where $d_i=\dim(\mathfrak{p}_i)$ and $A$ is some positive constant (see~\cite{Nikonorov2} for details).
In~\cite{Ab1,  AANS1, AANS2, Stat}  classifications of singular (equilibria) points of~\eqref{ricciflow} and their parametric bifurcations are studied on generalized Wallach spaces
and other classes of homogeneous spaces.
The interesting facts are
that  equilibria points are exactly Einstein metrics and vice versa, moreover,
the number of singular points and their types (hyperbolic, semi-hyperbolic, nilpotent or linearly zero) are most responsible for evolution of Riemannian metrics under NRF.  In~\cite{AN} we initiated the study of new aspects related to evolution of positively curved Riemannian metrics on GWS.

Historically the classification of homogeneous spaces admitting
homogeneous metrics with positive sectional curvature
were obtained as a result of a joint effort of
Aloff and Wallach~\cite{AW},
B\'erard Bergery~\cite{BB}, Berger~\cite{Be},
Wallach~\cite{Wal} and Wilking~\cite{Wil}.
Note  that B\'erard Bergery's classification odd-dimensional case
was recently revisited by Wolf and Xu in~\cite{XuWolf} and
a new version of the classification with refined proofs was
suggested by Wilking and Ziller in~\cite{WiZi}.
The classification of homogeneous spaces which admit
metrics with positive Ricci curvature was obtained by Berestovski\u\i\ in~\cite{Berest},
see also~\cite[Section 4.15]{BerNik2020}
where (in particular) homogeneous Riemannian manifolds of positive sectional curvature
and positive Ricci curvature are described.

Novadays an interesting question is to understand whether
the positivity of  sectional or (and) Ricci curvature of  Riemannian metrics
on a given Riemannian manifold is preserved or not when metrics are evolved by the Ricci flow.
Cheung and Wallach
showed  in Theorem~2 in~\cite{ChWal} non-maintenance of the sectional curvature for
certain metrics on the Wallach spaces
$$
W_6:=\operatorname{SU}(3)/T_{\max}, \qquad
W_{12}:=\operatorname{Sp(3)}/\operatorname{Sp(1)}\times \operatorname{Sp(1)}
\times \operatorname{Sp(1)}, \qquad
W_{24}:=F_4/\operatorname{Spin(8)}
$$
which are the only even-dimensional homogeneous spaces
admitting metrics with positive sectional curvature according to~\cite{Wal}.
B\"ohm and Wilking gave an example of metrics of positive sectional curvature
which can lose even the positivity of the Ricci curvature.
They proved in Theorem~3.1 in~\cite{Bo}  that
the (normalized) Ricci flow deforms some  metrics of positive sectional curvature
 on $W_{12}$  into metrics with mixed Ricci curvature. An analogous result was obtained for $W_{24}$ in Theorem~3 of~\cite{ChWal}.
In~\cite{Ab7, AN} we considered similar questions
on a class of generalized Wallach spaces with coincided parameters $a_1=a_2=a_3:=a\in (0,1/2)$
(see Section~\ref{defGWS} for definitions).
Note that a complete classification of generalized Wallach spaces
was obtained by Nikonorov~\cite{Nikonorov4}.
Chen et al. in~\cite{CKL} obtained similar results with simple $G$.

It should be also noted that
the Wallach spaces  $W_6, W_{12}$ and $W_{24}$ are exactly
partial cases of  generalized Wallach spaces with $a_1=a_2=a_3=a$
equal to $1/6, 1/8$ and $a=1/9$ respectively.

The following theorems were proved involving the apparatus of dynamical systems and
asymptotic analysis (by {\it generic} metrics we meant metrics satisfying  $x_i \ne x_j\ne x_k\ne x_i$,  $i,j,k\in \{1,2,3\}$):

\begin{theorem}[Theorem~1 in \cite{AN}]\label{the2}
On  the   Wallach spaces $W_6$, $W_{12}$, and $W_{24}$,  the normalized Ricci flow
evolves all generic metrics with positive sectional curvature into metrics with mixed sectional curvature.
\end{theorem}

Using the same tools the following results were obtained about the Ricci curvature:

\begin{theorem}[Theorem~3 in \cite{AN}]\label{Sect_Ricci_gen}
Let $G/H$ be a generalized Wallach space with $a_1=a_2=a_3=:a$, where $a\in (0,1/4)\cup(1/4,1/2)$.
If $a<1/6$, then
the normalized Ricci flow
evolves {\it all}\/ generic metrics  with positive Ricci curvature  into metrics with mixed Ricci curvature.
If $a\in (1/6,1/4)\cup(1/4,1/2)$, then the normalized Ricci flow evolves {\it all}\/ generic metrics into metrics with positive Ricci curvature.
\end{theorem}

\begin{theorem}[Theorem~4 in \cite{AN}]\label{Sect_Ricci_genn}
Let $G/H$ be a generalized Wallach space with $a_1=a_2=a_3=1/6$.
Suppose that it is supplied with the invariant Riemannian metric \eqref{metric} such that $x_k<x_i+x_j$ for all indices with $\{i,j,k\}=\{1,2,3\}$,
then the normalized Ricci flow on $G/H$ with this metric as the initial point, preserves the positivity of the Ricci curvature.
\end{theorem}

\begin{theorem}[Theorem~3 in \cite{Ab7}]\label{thm1}
The normalized Ricci flow on a generalized Wallach space $G/H$
$a_1=a_2=a_3=1/4$
evolves {\it all}\/ generic metrics into metrics  with positive Ricci curvature.
\end{theorem}

Recent time there are some new results on evolution of positively curved Riemannian metrics via Ricci flow. Among them Bettiol and Krishnan showed in~\cite{Bettiol} that Ricci flow does not preserve positive sectional curvature of Riemannian metrics on $S^4$ and $\mathbb{C}P^2$.
Some extended results on positive intermediate curvatures for an infinite series of homogeneous spaces
$
\operatorname{Sp}(n+1)/\operatorname{Sp}(n-1)\times \operatorname{Sp}(1)\times \operatorname{Sp}(1)
$
was obtained in~\cite{Gonzales}
by Gonz\'alez-\'Alvaro and  Zarei,
where the space $W_{12}$ is included as the first member corresponding to $n=2$.
The group of authors Cavenaghi et al. proved in~\cite{Caven} that $\operatorname{SU}(3)/T_{\max}$ admits metrics of positive intermediate
Ricci curvature losing positivity under Ricci flow.
They established similar properties  for
a family of homogeneous spaces
$
\operatorname{SU}(m+2p)/\operatorname{SU}(m)\times \operatorname{SU}(p)
\times \operatorname{SU}(p)
$
as well.
In this paper we consider
generalized Wallach spaces
$\operatorname{SO}(k+l+m)/\operatorname{SO}(k)\times \operatorname{SO}(l)
\times \operatorname{SO}(m)$,
$\operatorname{SU}(k+l+m)/\operatorname{SU}(k)\times \operatorname{SU}(l)
\times \operatorname{SU}(m)$
and
$\operatorname{Sp}(k+l+m)/\operatorname{Sp}(k)\times \operatorname{Sp}(l)
\times \operatorname{Sp}(m)$
and other ones listed in Table~1
in accordance with~\cite{Nikonorov4},
where  GWS~{\bf n} means a generalized Wallach space with $(\mathfrak{g}, \mathfrak{h})$ positioned in line~{\bf n}.
From the geometrical point of view only a countable number of natural triples $(k,l,m)$,
which correspond to actual generalized Wallach spaces,
can be interesting for us.
Knowing that  not every  triple $(a_1, a_2, a_3)\in (0,1/2)^3$
can correspond to some  generalized Wallach spaces,
nevertheless we will consider the problem in a more global context
for any reals  $a_1, a_2, a_3\in (0,1/2)$, using an opportunity
to involve fruitful methods of continuous mathematics.
Such an approach  justified itself introducing in~\cite{AANS1, AANS2}
a special algebraic surface defined by a polynomial of degree $12$ in three real variables
$a_1, a_2$ and $a_3$, which provides degenerate singular points to a three-dimensional
dynamical system obtained as a reduction of the normalized Ricci flow on generalized Wallach spaces.
Topological characteristics of that surface such as the number of connected components was found by~\cite{Ab2}, and a deeper structural analysis was given in~\cite{Bat2, Bat, Bruno, Bruno2}.
It should be also noted that in~\cite{Ab_RM}
results of Theorem~\ref{the2} were generalized
to the case $a\in (0,1/2)$ considering \eqref{three_equat} as an abstract dynamical system.
In the present paper we extend Theorems~\ref{Sect_Ricci_gen}\,--\,\ref{thm1}  to the case
of arbitrary  $a_1, a_2, a_3\in (0,1/2)$.
Our main results are contained in Theorems~\ref{thm_sum_a_i<1/2}\,--\,\ref{thm_3} and Corollary~\ref{corol_1}.

\begin{theorem}\label{thm_sum_a_i<1/2}
Let $G/H$ be a generalized Wallach space with $a_1+a_2+a_3\le 1/2$.
Then the normalized Ricci flow
evolves {\it some}\/  metrics~\eqref{metric}  with positive Ricci curvature  into metrics with mixed Ricci curvature.
\end{theorem}

\begin{theorem}\label{thm_sum_a_i>1/2}
Let $G/H$ be a generalized Wallach space with  $a_1+a_2+a_3> 1/2$.
\begin{enumerate}
\item
The normalized Ricci flow~\eqref{ricciflow}  evolves {\it all}\/
metrics~\eqref{metric}  with positive Ricci curvature  into metrics with positive Ricci curvature if~
$\theta \ge \max\left\{\theta_1, \theta_2, \theta_3\right\}$,
\item
At least some metrics~\eqref{metric}  with positive Ricci curvature  can be evolved into metrics with positive Ricci curvature, if  $\theta \ge \max\left\{\theta_1, \theta_2, \theta_3\right\}$ fails,
\end{enumerate}
where\, $\theta:=a_1+a_2+a_3-1/2\in (0,1)$, \,
$\theta_i:=a_i-1/2+\sqrt{(1-2a_i)(1+2a_i)^{-1}}\big/2$, \,
$i=1,2,3$.
\end{theorem}

Note that  $\theta \ge \max\left\{\theta_1, \theta_2, \theta_3\right\}$ is equivalent to  $4\left(a_j+a_k\right)^2\ge (1-2a_i)(1+2a_i)^{-1}$ announced in abstract
for all $\{i,j,k\}=\{1,2,3\}$.

For the spaces
$\operatorname{SO}(k+l+m)/\operatorname{SO}(k)\times \operatorname{SO}(l)\times \operatorname{SO}(m)$ (GWS~{\bf 1} in Table~1)
which satisfy~$a_1+a_2+a_3>1/2$ we proved
the following special theorem under the assumptions~$k\ge l\ge  m>1$
based on symmetry.

\renewcommand{\arraystretch}{1.6}
\begin{table}[h]\label{tab1}
{\bf Table 1}. The list of generalized Wallach spaces $G/H$ with  $G$ simple
according to~\cite{Nikonorov4}
\begin{center}
\begin{tabular}{l|l|l|c|c|c|rr}
\hline
\hline
 GWS&$\mathfrak{g}$& $\mathfrak{h}$ & $a_1$ & $a_2$ & $a_3$ & $\theta$ \\\hline
${\bf 1}$&$\mathfrak{so}(k+l+m)$& $\mathfrak{so}(k)\oplus\mathfrak{so}(l)\oplus\mathfrak{so}(m)$          &$\frac{k}{2(k+l+m-2)}$  & $\frac{l}{2(k+l+m-2)}$&  $\frac{m}{2(k+l+m-2)}$&$\frac{1}{k+l+m-2}$\\
${\bf 2}$&$\mathfrak{su}(k+l+m)$& $\mathfrak{su}(k)\oplus\mathfrak{su}(l)\oplus\mathfrak{su}(m)$          & $\frac{k}{2(k+l+m)}$  & $\frac{l}{2(k+l+m)}$&  $\frac{m}{2(k+l+m)}$&$0$   \\
{\bf 3}&$\mathfrak{sp}(k+l+m)$& $\mathfrak{sp}(k)\oplus\mathfrak{sp}(l)\oplus\mathfrak{sp}(m)$          &$\frac{k}{2(k+l+m+1)}$&$\frac{l}{2(k+l+m+1)}$&  $\frac{m}{2(k+l+m+1)}$&$\frac{-1}{2(k+l+m+1)}$\\
{\bf 4}&$\mathfrak{su}(2l),~ l\ge 2$& $\mathfrak{u}(l)$
& $\frac{l+1}{4l}$  & $\frac{l-1}{4l}$&  $\frac{1}{4}$ &$1/4$  \\
{\bf 5}&$\mathfrak{so}(2l),~ l\ge 4$& $\mathfrak{u}(1)\oplus \mathfrak{u}(l-1)$
& $\frac{l-2}{4(l-1)}$  & $\frac{l-2}{4(l-1)}$&  $\frac{1}{2(l-1)}$&$0$   \\
{\bf 6}&$\mathfrak{e}_6$& $\mathfrak{su}(4)\oplus 2\mathfrak{sp}(1)\oplus \mathbb{R}$
& $1/4$  & $1/4$&  $1/6$  &$1/6$ \\
{\bf 7}&$\mathfrak{e}_6$& $\mathfrak{so}(8)\oplus  \mathbb{R}^2$
& $1/6$  & $1/6$&  $1/6$  &$0$  \\
{\bf 8}&$\mathfrak{e}_6$& $\mathfrak{sp}(3)\oplus  \mathfrak{sp}(1)$
& $1/4$  & $1/8$&  $7/24$  &$1/6$ \\
{\bf 9}&$\mathfrak{e}_7$& $\mathfrak{so}(8)\oplus  3\mathfrak{sp}(1)$
& $2/9$  & $2/9$&  $2/9$ &$1/6$  \\
{\bf 10}&$\mathfrak{e}_7$& $\mathfrak{su}(6)\oplus \mathfrak{sp}(1)\oplus \mathbb{R}$
& $2/9$  & $1/6$&  $5/18$  &$1/6$ \\
{\bf 11}&$\mathfrak{e}_7$& $\mathfrak{s0}(8)$
& $5/18$  & $5/18$&  $5/18$  &$1/3$ \\
{\bf 12}&$\mathfrak{e}_8$& $\mathfrak{so}(12)\oplus  2\mathfrak{sp}(1)$
& $1/5$  & $1/5$&  $4/15$  &$1/6$ \\
{\bf 13}&$\mathfrak{e}_8$& $\mathfrak{so}(8)\oplus  \mathfrak{so}(8)$
& $4/15$  & $4/15$&  $4/15$  &$3/10$ \\
{\bf 14}&$\mathfrak{f}_4$& $\mathfrak{so}(5)\oplus  2\mathfrak{sp}(1)$
& $5/18$  & $5/18$&  $1/9$ &$1/6$  \\
{\bf 15}&$\mathfrak{f}_4$& $\mathfrak{so}(8)$
& $1/9$  & $1/9$&  $1/9$  &$-1/6$ \\
\hline\hline
\end{tabular}
\end{center}
\end{table}

\begin{theorem}\label{thm_3}
The following assertions hold  for $\operatorname{SO}(k+l+m)/\operatorname{SO}(k)\times \operatorname{SO}(l)\times \operatorname{SO}(m)$,
 $k\ge l\ge  m>1$ (see also Table~2):

\begin{enumerate}
\item
The normalized Ricci flow~\eqref{ricciflow}  evolves all metrics~\eqref{metric}  with $\operatorname{Ric}>0$
into metrics with $\operatorname{Ric}>0$   if either  $k\le 11$ or one of
the conditions  $2<l+m\le X(k)$ or  $l+m\ge Y(k)$ is satisfied for
$k\in \{12, 13,14,15,16\}$,
\item
The normalized Ricci flow~\eqref{ricciflow}  evolves some metrics~\eqref{metric} with $\operatorname{Ric}>0$  into metrics with $\operatorname{Ric}>0$ if
$X(k)<l+m<Y(k)$  for  $k\ge 12$.
\item
There exists a finite number  of $\operatorname{GWS}$~{\bf 1} on which any metric with $\operatorname{Ric}>0$ maintains  $\operatorname{Ric}>0$  under NRF~\eqref{ricciflow}.
But there are infinitely (countably)
many $\operatorname{GWS}$~{\bf 1} on which  $\operatorname{Ric}>0$ can be kept
at least for some metrics with $\operatorname{Ric}>0$,
\end{enumerate}
where
$X(k)=\frac{2k(k-2)}{k+2+\sqrt{k^2-12k+4}}-k+2$,
$Y(k)=\frac{2k(k-2)}{k+2-\sqrt{k^2-12k+4}}-k+2$ and
$2<X(k)<Y(k)$ for all $k\ge 12$.
\end{theorem}

\renewcommand{\arraystretch}{1.6}
\begin{table}[h]\label{tab1}
{\bf Table 2}.  Behavior of NRF~\eqref{ricciflow} on $\operatorname{GWS}$ {\bf 1}  for
$k\ge 12$ and $k\ge \max\{l,m\}$
\begin{center}
\begin{tabular}{c|c|c}
\hline
\hline
&$X(k)<l+m<Y(k)$& $2<l+m\le X(k)$~ or~ $l+m\ge Y(k)$ \\\hline
$12\le k\le 16$&Some metrics with $\operatorname{Ric}>0$ remain
$\operatorname{Ric}>0$ & All metrics with $\operatorname{Ric}>0$ remain
$\operatorname{Ric}>0$  \\
$k\ge 17$&Some metrics with $\operatorname{Ric}>0$ remain
$\operatorname{Ric}>0$ &   $-$     \\
\hline\hline
\end{tabular}
\end{center}
\end{table}

Theorems~\ref{thm_sum_a_i<1/2} and \ref{thm_sum_a_i>1/2} imply

\begin{corollary}\label{corol_1}
Under the normalized Ricci flow~\eqref{ricciflow} the positivity of the Ricci curvature
\begin{enumerate}
\item  is not preserved  on
            $\operatorname{GWS}$\; {\bf 2, 3, 5, 7, 15};
\item  is preserved on
            $\operatorname{GWS}$\; {\bf 4, 6, 8, 9, 10, 11, 12, 13, 14}.
\end{enumerate}
\end{corollary}

\section{Preliminaries}\label{defGWS}

In~\cite{Nikonorov2} the following explicit expressions
 $\operatorname{Ric}_{\bold{g}}=\left.{\bf r}_1\, \operatorname{Id} \right|_{\mathfrak{p}_1}+
\left.{\bf r}_2\, \operatorname{Id} \right|_{\mathfrak{p}_2}+
\left.{\bf r}_3\, \operatorname{Id} \right|_{\mathfrak{p}_3}$
and
$S_{\bold{g}}=d_1{\bf r}_1+d_2{\bf r}_2+d_3{\bf r}_3$
were found for the Ricci tensor $\operatorname{Ric}_{\bold{g}}$ and
the scalar curvature  $S_{\bold{g}}$  of the metric~\eqref{metric}
on  a generalized Wallach space~$G/H$,
where
$
{\bf r}_i:=\frac{1}{2x_i}+\frac{a_i}{2}\left(\frac{x_i}{x_jx_k}-\frac{x_k}{x_ix_j}-
\frac{x_j}{x_ix_k} \right)
$
are the principal Ricci curvatures,
$d_i$ are  the dimensions of the modules $\mathfrak{p}_i$ such that $d_1+d_2+d_3=n$ and
$a_i$ are real numbers in the interval $(0,1/2]$ for $\{i,j,k\}=\{1,2,3\}$.
Substituting the above mentioned expressions for $\operatorname{Ric}_{\bold{g}}$ and
 $S_{\bold{g}}$  into~\eqref{ricciflow} it can be reduced
 to the following system of three   ordinary differential equations
studied in~\cite{AANS1, AANS2}:
\begin{equation}\label{three_equat}
\dfrac {dx_1}{dt} = f_1(x_1,x_2,x_3), \quad
\dfrac {dx_2}{dt}=f_2(x_1,x_2,x_3), \quad
\dfrac {dx_3}{dt}=f_3(x_1,x_2,x_3),
\end {equation}
where $x_i=x_i(t)>0$ $(i=1,2,3)$, are parameters of the invariant metric \eqref{metric} and
\begin{eqnarray*}
f_1(x_1,x_2,x_3)&=&-1-a_1x_1 \left( \dfrac {x_1}{x_2x_3}-  \dfrac {x_2}{x_1x_3}- \dfrac {x_3}{x_1x_2} \right)+x_1B,\\
f_2(x_1,x_2,x_3)&=&-1-a_2x_2 \left( \dfrac {x_2}{x_1x_3}- \dfrac {x_3}{x_1x_2} -  \dfrac {x_1}{x_2x_3} \right)+x_2B,\\
f_3(x_1,x_2,x_3)&=&-1-a_3x_3 \left( \dfrac {x_3}{x_1x_2}-  \dfrac {x_1}{x_2x_3}- \dfrac {x_2}{x_1x_3} \right)+x_3B,
\end{eqnarray*}
$$
B:=\left(\sum_{i=1}^3 \dfrac {1}{a_ix_i}- \sum_{\substack {\{i,j,k\}=\{1,2,3\}\\ j< k}}
\dfrac {x_i}{x_jx_k}  \right)
\left(\sum_{i=1}^3a_i^{-1} \right)^{-1}.
$$

\medskip

Recall that the function  $\operatorname{Vol}=c$ is a first integral of the system~\eqref{three_equat} for any $c>0$,
where
$
\operatorname{Vol}:=x_1^{1/a_1}x_2^{1/a_2}x_3^{1/a_3}
$
(see~\cite{Ab_RM}). The surface  $\operatorname{Vol}=1$ is  important
which corresponds to parameters of unite volume metrics. Denote it $\Sigma$.
The case~$\operatorname{Vol}=c$ could actually be reduced to the case  $c=1$
using homogeneity of the functions $f_i$ and autonomy property in~\eqref{three_equat}.
This could be made by introducing new variables $x_i=\widetilde{x}_i\sqrt[3]{c}$ and $t=\tau\sqrt[3]{c}$, then
we get a new system of ODEs in  new functions $\widetilde{x}_i(\tau)$ but of the same form as the original one~\eqref{three_equat} (see also \cite{Ab_RM}).
Therefore we can consider metrics with $\operatorname{Vol}=1$ without loss of generality.
Being  a first integral of the system~\eqref{three_equat}
the identity $\operatorname{Vol}=1$ allows to reduce~\eqref{three_equat}  to the following system of two ODEs
(see details in~\cite{AANS1}):
 \begin{equation}\label{two_equat}
\dfrac {dx_1}{dt} = \widetilde{f}_1(x_1,x_2), \qquad
\dfrac {dx_2}{dt}=\widetilde{f}_2(x_1,x_2),
\end {equation}
where
$\widetilde{f}_i(x_1,x_2):=f_i(x_1,x_2, \varphi(x_1,x_2))$,\;
$\varphi(x_1,x_2):=x_1^{-a_3/a_1}x_2^{-a_3/a_2}$,\; $i=1,2$.

Introduce   surfaces $\Lambda_i$ in $(0,+\infty)^3$ defined  by the equations
$\lambda_i:=a_i\left(x_i^2-x_j^2-x_k^2\right)+x_jx_k=0$ (equivalent to $\bold{r}_i=0$)
for    $\{i,j,k\}=\{1,2,3\}$. Denote by $R$ the set
of points $(x_1,x_2,x_3)\in (0,+\infty)^3$ satisfying the inequalities $\lambda_i>0$ for all $i=1,2,3$.
Then $\operatorname{Ric}_{\bold{g}}>0$ is equivalent to  $(x_1,x_2,x_3)\in R$.

\section{Auxiliary results}

Introduce the parameter
$\omega:=\left(a_i^{-1}+a_j^{-1}+a_k^{-1}\right)^{-1}$.
We will often use real roots
\begin{equation}\label{m_and_M}
m_i=m(a_i):=\frac{1-\sqrt{1-4a_i^2}}{2a_i} \mbox{~~ and ~~}
M_i=M(a_i):=\frac{1+\sqrt{1-4a_i^2}}{2a_i}
\end{equation}
 of the quadratic equation
$t^2-a_i^{-1}t+1=0$ which depend on the parameter~$a_i$.
Clearly $M_i=m_i^{-1}$.
We also need the following obvious inequalities for all $a_i\in (0,1/2)$:
\begin{equation}\label{tlar}
0<a_i<m_i<2a_i<1<\frac{1}{2a_i}<M_i<\frac{1}{a_i}.
\end{equation}

\subsection{Structural properties of some surfaces and curves
related to the set of metrics $\operatorname{Ric}>0$}

\begin{figure}[h]
\centering
\includegraphics[width=0.9\textwidth]{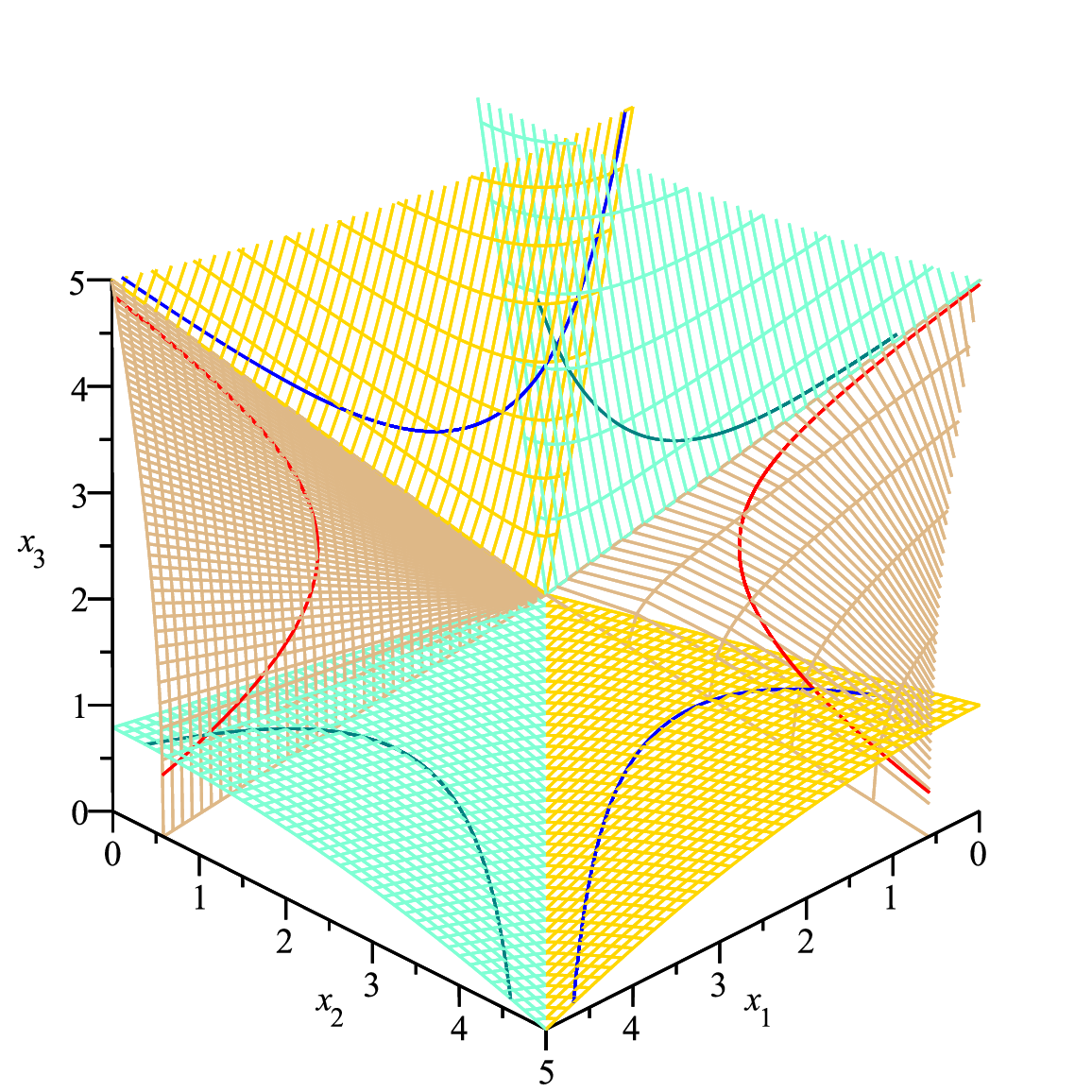}
\caption{The cones $\Lambda_1, \Lambda_2$ and $\Lambda_3$ (in yellow, aquamarine and burlywood colors) at $a_1=5/26$, $a_2=4/26$, $a_3=3/26$}
\label{pic10}
\end{figure}

\begin{lemma}\label{Delta_Conic_gen}
For an arbitrary generalized Wallach space with $a_1,a_2,a_3\in (0,1/2)$ the set $R$
is bounded by the conic surfaces
$\Lambda_1$, $\Lambda_2$ and $\Lambda_3$,
each consisting of two connected components.
Moreover,  for all $a_i,a_j\in (0,1/2)$
the cones  $\Lambda_i$ and $\Lambda_j$
intersect each other along two straight lines:
the first of them  is defined  by $x_i=x_j$, $x_k=0$,
and the second one has an equation
\begin{equation}\label{commonline}
x_i=up, \qquad x_j=vp, \qquad x_k=p>0, \qquad \{i,j,k\}=\{1,2,3\},
\end{equation}
where
\begin{equation}\label{alpha_beta}
\begin{array}{l}
u=\Phi(a_i,a_j):=
\begin{cases}
\phi(a_i,a_j),  &\mbox{if}
\quad a_i\ne  a_j,\\
a_j,  &\mbox{if} \quad a_i=a_j,
\end{cases} \\
v=\Psi(a_i,a_j):=
\begin{cases}
\psi(a_i,a_j),  &\mbox{if}
\quad a_i\ne a_j,\\
a_i,  &\mbox{if} \quad a_i=a_j,
\end{cases}
\end{array}
\end{equation}
with~ $\phi(a_i,a_j):=\dfrac{1}{2}\dfrac{4a_i^2-1+\sqrt{\Delta_{ij}}}{a_i^2-a_j^2}\, a_j>0$, \;
$\psi(a_i,a_j):=\dfrac{1}{2}\dfrac{4a_j^2-1+\sqrt{\Delta_{ij}}}{a_j^2-a_i^2}\, a_i>0$ and
$\Delta_{ij}:=\left(1-4a_i^2\right)\left(1-4a_j^2\right)>0$.
See Fig.~\ref{pic10} for illustrations.
\end{lemma}

\begin{proof}
The inequality  $\lambda_i=a_i(x_i^2-x_j^2-x_k^2)+x_jx_k>0$ is equivalent to
$x_i^2>T:=x_j^2-a_i^{-1}x_jx_k+x_k^2$.
Any $x_i>0$ satisfies $\lambda_i>0$ if  $T\le 0$.
If  $T>0$ then  $\lambda_i>0$  is equivalent to $x_i> \sqrt{T}$.
Considering $T>0$ as a quadratic inequality with respect to $x_j$
which admits solutions   $0<x_j<m_ix_k$ or $x_j>M_ix_k$
we conclude that the set~$R$ is bounded by the coordinate planes $x_1=0$, $x_2=0$, $x_3=0$
and by two components of the cone $\Lambda_i$,
defined by the same equation $x_i=\sqrt{T}$ in two different domains
$\left\{x_k>0, \, 0<x_j<m_ix_k\right\}$ and
$\left\{x_k>0, \, x_j>M_ix_k\right\}$ of the coordinate plane $(x_k,x_j)$.

Due to symmetry analogous assertions can be obtained for the surfaces
$\Lambda_j$ and $\Lambda_k$  using permutations of indices.
Thus  we established that $R$ is bounded by the cones $\Lambda_1$, $\Lambda_2$ and $\Lambda_3$.
To find common points of the surfaces $\Lambda_i$ and $\Lambda_j$
consider the system of equations
$\lambda_i=a_i\left(x_i^2-x_j^2-x_k^2\right)+x_jx_k=0$ and
$\lambda_j=a_j\left(x_j^2-x_i^2-x_k^2\right)+x_ix_k=0$.
We can pass to a new system
$a_i(u^2-v^2-1)+v=0$, $a_j(v^2-u^2-1)+u=0$,
where $x_i=ux_k$, $x_j=vx_k$.
It follows then  $a_iu+a_jv-2a_ia_j=0$ or
$v=a_ia_j^{-1}(2a_j-u)$.
Substituting this expression into one of the  previous equations,
we obtain a quadratic equation with respect to~$u$:
$$
\left(a_i^2-a_j^2\right)u^2-\left(4a_i^2-1\right)a_ju+\left(4a_i^2-1\right)a_j^2=0.
$$

The case  $a_i=a_j$ leads  to a solution $u=a_j$. It follows then $v=a_i$.

Assume that  $a_i\ne a_j$.
{\it The case $a_i>a_j$}. Then $4a_i^2-1+\sqrt{\Delta_{ij}}>0$ and
the latter quadratic equation admits a positive root $u^{(1)}$,
equal to  $\phi(a_i,a_j)$ shown in~\eqref{alpha_beta}.
Another root $u^{(2)}$, associated with the summand $-\sqrt{\Delta_{ij}}$
is out of our interest because of
$4a_i^2-1-\sqrt{\Delta_{ij}}=-(1-4a_i^2+\sqrt{\Delta_{ij}})<0$
independently on $a_i$ and $a_j$.
Substituting $u^{(1)}=\phi(a_i,a_j)$
into
$
v^{(1)}=a_ia_j^{-1}\left(2a_j-u^{(1)}\right)
$
gives~$v^{(1)}=\psi(a_i,a_j)$ in~\eqref{alpha_beta}.

{\it The case $a_i<a_j$}. In this case  $u^{(1)}>0$ since $4a_i^2-1+\sqrt{\Delta_{ij}}<0$.
The second root~$u^{(2)}$ we reject again although it is positive,
because the corresponding~$v^{(2)}$  occurs negative.

Summing   $a_j\lambda_i=0$ and  $a_i\lambda_j=0$
we can find another family of solutions $x_i=x_j$, $x_k=0$
of the system of $\lambda_i=0$ and $\lambda_j=0$
which will also be useful in sequel.
It should be noted that there are also negative solutions of the system of $\lambda_i=0$ and $\lambda_j=0$
non permissible for our aims.
Lemma~\ref{Delta_Conic_gen} is proved.
\end{proof}

\begin{lemma}\label{alpa_beta_ineq}
Assume that $0<a_j<a_i<1/2$. Then  $u$ and
$v$ in~\eqref{alpha_beta}
satisfy $0<\widetilde{u} < 1 < \widetilde{v}$,
where $\widetilde{u}:=u a_j^{-1}$ and $\widetilde{v}:=v a_i^{-1}$.
\end{lemma}

\begin{proof}
For all $a_i\ne a_j$ both of the inequalities
$\widetilde{u}<1$ and
$\widetilde{v}>1$ are equivalent to
$\sqrt{\Delta_{ij}}<1-2\left(a_i^2+a_j^2\right)$ or the same
$(a_i-a_j)^2(a_i+a_j)^2>0$ after raising to square
based on   $1-2\left(a_i^2+a_j^2\right)>0$.
Lemma~\ref{alpa_beta_ineq} is proved.
\end{proof}

\medskip

Denote by $\Lambda_{ij}$ and $\Lambda_{ji}$  connected components of the cones  $\Lambda_{i}$ and $\Lambda_{j}$ respectively,
which contain the common straight line~\eqref{commonline}.
For  components of $\Lambda_{i}$ and $\Lambda_{j}$ meeting each other along  $x_i=x_j$, $x_k=0$
accept  notations $\Lambda_{ik}$ and $\Lambda_{jk}$ respectively.
Then each curve $r_i$ formed as   $\Sigma \cap \Lambda_i$
consists of two connected  components
$r_{ij}:=\Sigma \cap \Lambda_{ij}$ and $r_{ik}:=\Sigma\cap \Lambda_{ik}$.
By analogy  $r_j:=\Sigma\cap \Lambda_j=r_{ji}\cup r_{jk}$, where
$r_{ji}:=\Sigma \cap \Lambda_{ji}$ and $r_{jk}:=\Sigma\cap \Lambda_{jk}$.

\begin{lemma}\label{param_r_i}
For an arbitrary generalized Wallach space with $a_1,a_2,a_3\in (0,1/2)$
each curve $r_i=\Lambda_i \cap \Sigma$
consists of two connected components $r_{ij}$ and $r_{ik}$, which are smooth curves
and can be represented by parametric  equations
\begin{equation}\label{parcurvegen}
\begin{array}{l}
x_i=\phi_i (t):=t^{-\omega a_j^{-1}}
\left(\sqrt{t^2-a_i^{-1}t+1}\right)^
{\omega \left(a_j^{-1}+a_k^{-1}\right)},\\
x_j= \phi_j (t):= tx_k,\\
x_k= \phi_k (t):=\left(\sqrt{t^2-a_i^{-1}t+1}\right)^{-1}x_i
\end{array}
\end{equation}
defined for $t\in (0,m_i)$ and $t\in (M_i,+\infty)$ respectively,
where $\omega=\left(a_i^{-1}+a_j^{-1}+a_k^{-1}\right)^{-1}$ and $\{i,j,k\}=\{1,2,3\}$.
In addition,
\begin{equation}\label{parcurlims}
\begin{array}{l}
\lim\limits_{t\to 0+} x_i=\lim\limits_{t\to 0+} x_k = +\infty, \; \lim\limits_{t\to 0+} x_j =0,
\;
\lim\limits_{t\to m_i-} x_j= \lim\limits_{t\to m_i-} x_k = +\infty, \; \lim\limits_{t\to m_i-} x_i=0,
\\
 \lim\limits_{t\to M_i+} x_j= \lim\limits_{t\to M_i+} x_k = +\infty, \; \lim\limits_{t\to M_i+} x_i=0,
\;
\lim\limits_{t\to +\infty} x_i=\lim\limits_{t\to +\infty} x_j = +\infty, \; \lim\limits_{t\to +\infty} x_k =0.
\end{array}
\end{equation}
\end{lemma}

\begin{proof}
Substituting  $x_j=tx_k$ into the equation $a_i(x_i^2-x_j^2-x_k^2)+x_jx_k=0$
of~$\Lambda_i$ and solving the quadratic equation
$
\left(t^2-a_i^{-1}t+1\right)x_k^2=x_i^2
$
with respect to $x_k>0$, we obtain the third expression in~\eqref{parcurvegen}.
Then the expression for $x_i$ is obtained  using
$\operatorname{Vol}=x_i^{1/a_i}x_j^{1/a_j}x_k^{1/a_k}=1$.

As it follows from Lemma~\ref{Delta_Conic_gen}, the curve  $r_i=\Sigma \cap \Lambda_i$
consists of two components $r_{ij}=\Sigma \cap \Lambda_{ij}$ and  $r_{ik}=\Sigma \cap \Lambda_{ik}$. This fact  is also clear from
$t^2-a_i^{-1}t+1=\left(t-m_i\right)\left(t-M_i\right)>0$,
$t\in (0,m_i)\cup (M_i,+\infty)$,
where $m_i=m(a_i)$ and $M_i=M(a_i)$ are functions given in~\eqref{m_and_M}.
The components $r_{ij}$ and $r_{ik}$ are connected curves
being continuous images of the connected sets $(0,m_i)$ and  $(M_i, +\infty)$.
Moreover since $x_i=\phi_i(t)$, $x_j=\phi_j(t)$ and $x_k=\phi_k(t)$
are differentiable functions of $t\in (0,m_i) \cup (M_i,+\infty)$,
the components  $r_{ij}$ and $r_{ik}$ are smooth curves.
Recall that analogous parametric equations
can be obtained for the curves
$r_j=\Sigma \cap \Lambda_j$ and $r_k=\Sigma \cap \Lambda_k$
using permutations of the indices  $i,j$ and $k$ in~\eqref{parcurvegen}.

Let us study behavior of the curves $r_{ij}$ and $r_{ik}$
in neighborhoods of the boundaries of the intervals $(0,m_i)$ and  $(M_i, +\infty)$.
Clearly $x_i(t)\to +\infty$ as~$t\to 0+$  and $x_i(t)\to 0$
as  $t\to m_i-$ or $t\to M_i+$.
Analogously
$
x_i=\phi_i(t)\sim t^{-\omega a_j^{-1}} t^{\omega\big(a_j^{-1}+a_k^{-1}\big)}
=t^{\omega a_k^{-1}} \to +\infty
$
as $t\to +\infty$.
Since $1-\omega\big(a_j^{-1}+a_k^{-1}\big)>0$,
 we
easily can find the limits in~\eqref{parcurlims} for
$
x_j=\phi_j(t) = t^{1-\omega a_j^{-1}}
\left(\sqrt{\left(t-m_i\right)\left(t-M_i\right)}\right)^{-1+\omega\big(a_j^{-1}+a_k^{-1}\big)}
$
and
$x_k=\phi_k(t)= t^{-\omega a_j^{-1}}
\left(\sqrt{\left(t-m_i\right)\left(t-M_i\right)}\right)^{-1+\omega\big(a_j^{-1}+a_k^{-1}\big)}
$.
Lemma~\ref{param_r_i} is proved.
\end{proof}

\begin{figure}[h]
\centering
\includegraphics[width=0.9\textwidth]{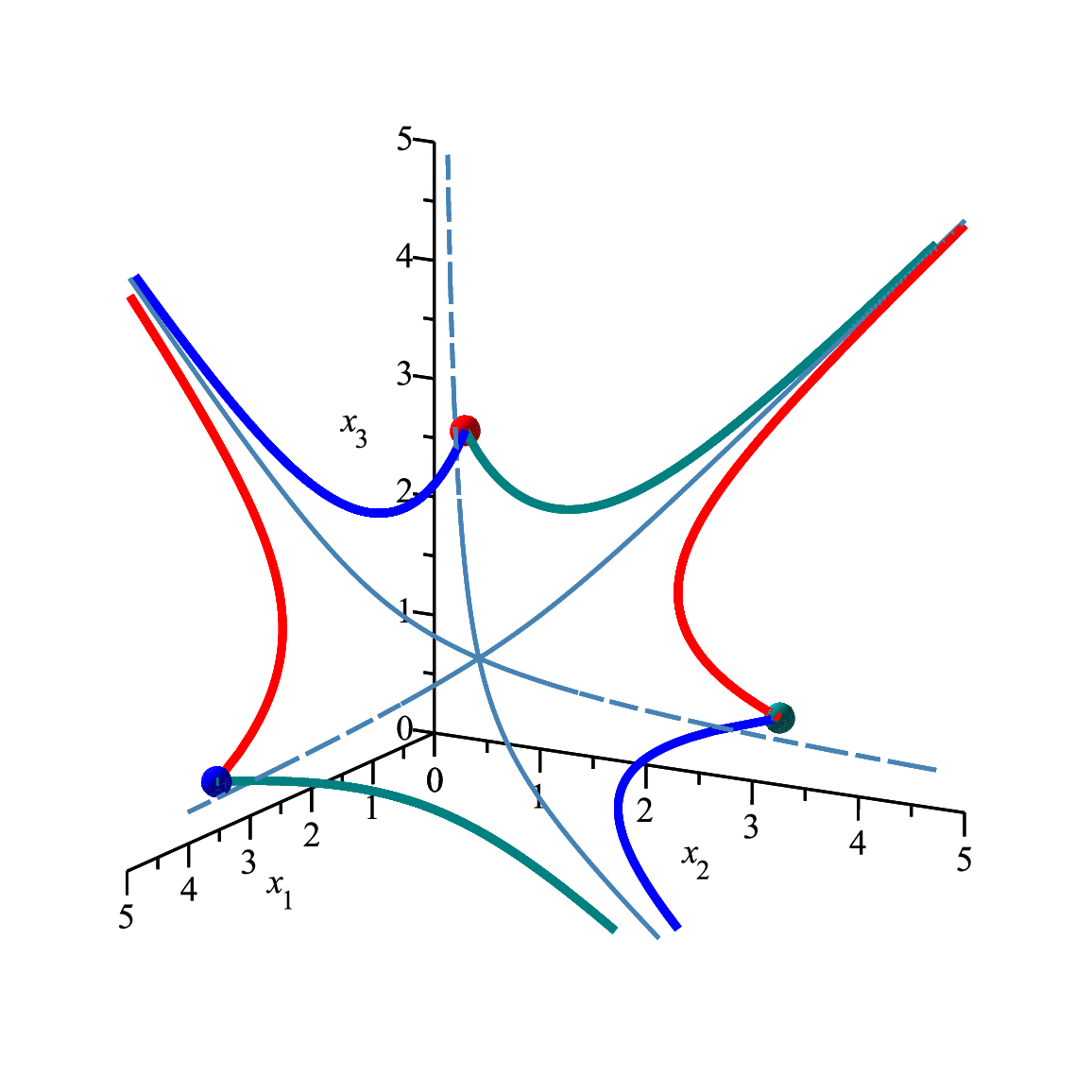}
\caption{The curves $r_1, r_2, r_3$
(in blue, teal, red colors) after removing extra pieces, the curves $I_1, I_2, I_3$ (steel blue dash lines) and the points $P_{12}, P_{13}, P_{23}$ (in red, teal, blue colors) at $a_1=5/26$, $a_2=4/26$, $a_3=3/26$}
\label{pic1}
\end{figure}

\begin{lemma}\label{r_and_I}
For each pair of the curves $r_i$ and $r_j$
corresponding to an arbitrary generalized Wallach space with $a_1,a_2,a_3\in (0,1/2)$
the curve $I_k$ defined by the equations
\begin{equation}\label{Eq_I_k}
x_i=x_j=\tau>0, \qquad x_k=\tau^{-a_k\left(a_i^{-1}+ a_j^{-1}\right)},
\qquad \{i,j,k\}=\{1,2,3\},
\end{equation}
does not intersect their  components $r_{ik}$ and $r_{jk}$,
but  approximates them as accurately as desired at infinity,
at the same time $I_k$ intersects both  $r_i$ and $r_j$ on their components
 $r_{ij}$ and $r_{ji}$. In addition,~$I_k$ can not intersect the curve $r_k$,
 moreover, $r_k$ moves away from $I_k$ in both directions
as $\tau \to 0+$ and $\tau \to +\infty$.
See Fig.~\ref{pic1} for illustrations.
\end{lemma}

\begin{proof} $I_k$ is exactly the curve that belongs to the surface~$\Sigma$ and have a projection $x_i=x_j$ on the coordinate plane $x_k=0$. Substituting $x_i=x_j=\tau$ into the equation
$x_i^{1/a_i}x_j^{1/a_j}x_k^{1/a_k}=1$ of~$\Sigma$  we find $x_k$ given in~\eqref{Eq_I_k}.
Substituting the expressions~\eqref{Eq_I_k} into the functions
$\lambda_i=a_i\left(x_i^2-x_j^2-x_k^2\right)+x_jx_k$
and
$\lambda_j=a_j\left(x_j^2-x_i^2-x_k^2\right)+x_ix_k$
(recall that the cones $\Lambda_i$ and $\Lambda_j$ are defined by $\lambda_i=0$ and $\lambda_j=0$
respectively) we obtain
$$
\lambda_i=\lambda_i(\tau)=\frac{1-a_i\tau^{-1-a_k\left(a_i^{-1}+ a_j^{-1}\right)}}{2\tau}, \qquad \lambda_j=\lambda_j(\tau)=\frac{1-a_j\tau^{-1-a_k\left(a_i^{-1}+ a_j^{-1}\right)}}{2\tau}.
$$
In what follows that~$I_k$ intersects the curve~$r_i=\Sigma\cap \Lambda_i$ at the unique point with coordinates
$x_i^0=x_j^0=\tau_0$, $x_k^0=\tau_0^{-a_k\left(a_i^{-1}+ a_j^{-1}\right)}$
where
$\tau_0=a_i^{\left(1+a_k\left(a_i^{-1}+a_j^{-1}\right)\right)^{-1}}$
is the unique root
of the equation~$\lambda_i(\tau)=0$.
Using Lemma~\ref{param_r_i}, write out the equation
$\tau_0=x_j^0=tx_k^0=t\tau_0^{-a_k\left(a_i^{-1}+ a_j^{-1}\right)}$
admitting the unique root
$t=t_k:=\tau_0^{1+a_k\left(a_i^{-1}+a_j^{-1}\right)}=a_i$,
at which the mentioned point $I_k\cap r_i$  belongs to~$r_i$.
The inequalities $0<t_k=a_i<m_i$ (see~\eqref{tlar}) imply
that $I_k$ must intersect $r_i$ on its  component~$r_{ij}$.
Since such an intersection with~$r_i$  happens only once,
we conclude that $I_k$  can not intersect the component~$r_{ik}$.
Moreover, the value of the limit
$
\lim_{\tau \to +\infty}\lambda_i=\frac{1}{2}\lim_{\tau \to +\infty}
\left(\tau^{-1}-a_i\tau^{-2-a_k\left(a_i^{-1}+ a_j^{-1}\right)}\right)=0
$
shows that $r_{ik}$ can be approximated by~$I_k$ at infinity as accurately as desired.

From the analysis of the function $\lambda_j=\lambda_j(\tau)$, similar conclusions will follow regarding the components~$r_{ji}$ and~$r_{jk}$ of the curve~$r_j$.
For the curve  $r_k$ we have
$
\lambda_k(\tau)=\dfrac{a_k\tau^{-2a_k\left(a_i^{-1}+ a_j^{-1}\right)}
+(1-2a_k)\tau^2}
{2\tau^{2-a_k\left(a_i^{-1}+ a_j^{-1}\right)}}
$.
The equation $\lambda_k(\tau)=0$ can not admit any solution
due to positiveness of the numerator of the  latter fraction.
Therefore $I_k$ never cross $r_k$.
Moreover, they move away from each other as much as desired in both directions:
\begin{eqnarray*}
\lim_{\tau \to +\infty} \lambda_k(\tau)=
\frac{a_k}{2}\lim_{\tau \to +\infty} \tau^{-2-a_k\left(a_i^{-1}+ a_j^{-1}\right)}+
\frac{1-2a_k}{2}\lim_{\tau \to +\infty} \tau^{a_k\left(a_i^{-1}+ a_j^{-1}\right)}=0+\infty=+\infty,\\
\lim_{\tau \to 0+} \lambda_k(\tau)=
\frac{a_k}{2}\lim_{\tau \to 0+} \tau^{-2-a_k\left(a_i^{-1}+ a_j^{-1}\right)}+
\frac{1-2a_k}{2}\lim_{\tau \to 0+} \tau^{a_k\left(a_i^{-1}+ a_j^{-1}\right)}=+\infty+0=+\infty.
\end{eqnarray*}
Lemma~\ref{r_and_I} is proved.
\end{proof}

\begin{lemma}\label{pairwise_inter_r}
For any generalized Wallach space with $a_1,a_2,a_3\in (0,1/2)$ the following assertions hold:
\begin{enumerate}
\item
For each pair of the curves~$r_i$  and~$r_j$ (respectively $r_k$)
their components $r_{ij}$ and $r_{ji}$ (respectively $r_{ik}$ and $r_{ki}$)
admits  exactly one common point~$P_{ij}$ (respectively $P_{ik}$)
with coordinates $x_i=\alpha p^0$, $x_j=\beta p^0$,  $x_k=p^0$
(respectively $x_i=\gamma q^0$,  $x_k=\delta q^0$, $x_j= q^0$),
where
\begin{equation}\label{Pij}
\begin{array}{l}
p^0:=\alpha^{-\omega a_i^{-1}}\beta^{-\omega a_j^{-1}}, \quad
\alpha:=\Phi(a_i,a_j),  \quad \beta:=\Psi(a_i,a_j), \\
q^0:=\gamma^{-\omega a_i^{-1}}\delta^{-\omega a_k^{-1}}, \quad
\gamma:=\Phi(a_i,a_k), \quad \delta:=\Psi(a_i,a_k)
\end{array}
\end{equation}
with the functions $\Phi$ and $\Psi$ defined in~\eqref{alpha_beta},
$\{i,j,k\}=\{1,2,3\}$,
whereas their other components $r_{ik}$ and $r_{jk}$
(respectively $r_{ij}$ and $r_{kj}$)
are disjoint approximating each other at infinity as close as desired
(see Fig.~\ref{pic1}).

\item
The values  $t_{ij}$ and $t_{ik}$ of the parameter $t$,
which correspond to the points  $P_{ij}$
and $P_{ik}$ at the parametrization~\eqref{parcurvegen},
can be found by the formula
\begin{equation}\label{t_ij_a}
t_{ij}=\beta, \qquad t_{ik}=\delta^{-1}.
\end{equation}
In addition $a_i\le t_{ij}<m(a_i)$ and  $M(a_i)<t_{ik}\le a_i^{-1}$
for all $a_i, a_j, a_k\in (0,1/2)$.
\end{enumerate}
\end{lemma}

\begin{proof} $(1)$
By definition  $r_i=\Lambda_i \cap \Sigma$ and $r_j=\Lambda_j \cap \Sigma$.
It follows then  $r_i\cap r_j=(\Lambda_i\cap \Lambda_j)\cap \Sigma$, and hence,
the problem of finding  $r_i\cap r_j$
reduces to searching for possible common points
of the straight line $\Lambda_i\cap \Lambda_j$ with the surface $\Sigma$.
As it was proved in Lemma~\ref{Delta_Conic_gen},
the intersection $\Lambda_i\cap \Lambda_j$ equal to $\Lambda_{ij}\cap \Lambda_{ji}$ is the straight line
 $x_i=\alpha p$, $x_j=\beta p$,  $x_k=p$.
Substituting these expressions into the equation
$x_i^{1/a_i}x_j^{1/a_j}x_k^{1/a_k}=1$ of~$\Sigma$ we obtain
the following equation regarding to the parameter~$p$:
$$
\alpha^{1/a_i} \beta^{1/a_j} p^{1/a_i+1/a_j+1/a_k}=1.
$$
This yields coordinates
$x_i=\alpha p^0$, $x_j=\beta p^0$,  $x_k=p^0$ of an unique point
which belong to $r_{ij}\cap r_{ji}=(\Lambda_{ij}\cap \Lambda_{ji})\cap \Sigma$
for  $\alpha, \beta$ and~$p^0$ exposed in~\eqref{Pij}. Let us denote that point by~$P_{ij}$.
Thus  $\{P_{ij}\}=r_{ij}\cap r_{ji}=r_i\cap r_j$.
Repeating the same reasoning for the curves $r_i$ and $r_k$,  we find coordinates
$x_i=\gamma q^0$,  $x_k=\delta q^0$, $x_j= q^0$
of the unique point $P_{ik}\in r_{ik}\cap r_{ki}=r_i\cap r_k$.

In order to highlight disjoint components of the curves $r_i$ and $r_j$,
recall another  common line $x_i=x_j$, $x_k=0$
of the cones $\Lambda_i$  and $\Lambda_j$
(more precisely, the common line of their components $\Lambda_{ik}$ and $\Lambda_{jk}$),
mentioned in Lemma~\ref{Delta_Conic_gen}.
The  components $r_{ik}$ and $r_{jk}$ are situated on opposite sides of~$I_k$, therefore
they can not admit a common point. Moreover, $r_{ik}$ and $r_{jk}$ approximate
the same curve~$I_k$ at infinity by Lemma~\ref{r_and_I}, therefore at infinity they
approximate each other.

\smallskip

$(2)$ Assume that $r_i$  is parameterized as in Lemma~\ref{param_r_i}.
Let  $t=t_{ij}$ be a value of $t$ which corresponds to the point $P_{ij}$
on  $r_{ij}$.
The coordinates of $P_{ij}$ are $x_i=\alpha p^0$, $x_j=\beta p^0$, $x_k=p^0$
as we established.
Then the equalities  $\phi_j(t)=\beta p^0$,  $\phi_k(t)=p^0$ and
$\phi_j(t)=t\phi_k(t)$ in~\eqref{parcurvegen} easily imply
$t_{ij}=\beta:=\Psi(a_i,a_j)$.

Now we claim that   $t_{ij}\in [a_i, m_i)$ for all $a_i, a_j\in (0,1/2)$.

Indeed if $a_i=a_j$ for  $i\ne j$ then $t_{ij}=a_i$ obviously.

If $a_i\ne a_j$ then assume  $a_i>a_j$ without loss of generality.
Then   Lemma~\ref{alpa_beta_ineq} implies $t_{ij}>a_i$.
The inequality  $t_{ij}<m_i$ is equivalent to
\begin{multline*}
\left(1-4a_j^2\right)a_i^2-a_i^2\sqrt{\Delta_{ij}}<\left(a_i^2-a_j^2\right)
\left(1-\sqrt{1-4a_i^2}\right) \\ \Leftrightarrow~ a_i^2-a_j^2<a_i^2\sqrt{1-4a_j^2}-a_j^2\sqrt{1-4a_i^2}
~\Leftrightarrow~
\left(a_i^2-a_j^2\right)^2<\left(a_i^2\sqrt{1-4a_j^2}-a_j^2\sqrt{1-4a_i^2}\right)^2\\
~\Leftrightarrow~ \sqrt{\Delta_{ij}}<1-2\left(a_i^2+a_j^2\right)
~\Leftrightarrow~
(a_i-a_j)^2(a_i+a_j)^2>0.
\end{multline*}

Thus $t_{ij}\in (a_i, m_i)$ for $a_i\ne a_j$.
Coordinates $x_i=\gamma q^0$,  $x_k=\delta q^0$, $x_j= q^0$ of the point $P_{ik}$
can be clarified by the same way.
The value
$
t_{ik}=\delta^{-1}=\frac{1}{\Psi(a_i,a_k)}
$
is obtained as an unique root of the equation  $q^0=\phi_j(t)=t\phi_k(t)=t\delta q^0$ (see~\eqref{alpha_beta}).

Analogously  $t_{ik}\in \left(M_i,  a_i^{-1}\right]$.
Indeed if $a_i=a_k$ then $t_{ik}=\Psi(a_i,a_k)^{-1}=a_i^{-1}>M_i$
due to~\eqref{tlar}.

If $a_i\ne a_k$ then assume $a_i>a_k$.
The inequality  $t_{ik}<a_i^{-1}$ is clear from Lemma~\ref{alpa_beta_ineq}.
The inequality $t_{ik}>M_i$
can be established in similar way as above.
Lemma~\ref{pairwise_inter_r} is proved.
\end{proof}

\begin{remark}\label{sec_tails}
It is advisable to remove extra pieces of the curves $r_1$, $r_2$ and $r_3$,
retaining after their pairwise intersections.
In fact those extra pieces do not belong to the boundary of the domain $R$.
Let us consider the example of the curve $r_1$ which
intersects~$r_2$ on the unique point
$P_{12}(\alpha p^0, \beta p^0, p^0)$ known from Lemma~\ref{pairwise_inter_r}, where   $\alpha, \beta$ and $p^0$
are defined by formulas~\eqref{alpha_beta} and~\eqref{Pij} at $i=1, j=2$.
An analysis of \eqref{parcurvegen} and \eqref{parcurlims} shows that
the interval   $[t_{12}, m_1)$ corresponds to the extra piece
of $r_{12}$ behind the point $P_{12}$,
where $t_{12}$ can be found by~\eqref{t_ij_a}.
Therefore $(0,t_{12}]$ should be taken as an interval of parametrization
for the component $r_{12}$
instead of the original interval $(0,m_1)$.
Similarly, the interval $(M_1,+\infty)$  corresponding to~$r_{13}$
 before its intersection with $r_{31}$ at the point
$P_{13}(\gamma q^0, q^0, \delta q^0)$, will be reduced to $[t_{13}, +\infty)$,
where  $\gamma, \delta$ and $q^0$ are defined by formulas~\eqref{alpha_beta}
and~\eqref{Pij} at $i=1, j=3$.
\end{remark}

\subsection{Normals to cones
and the vector field associated with the differential system}

To study the behavior of trajectories of the system~\eqref{three_equat}
 we need the sign of the standard inner product
$$
({\bold V}, \nabla \lambda_i)=f_i\,\frac{\partial \lambda_i}{\partial x_i}+f_j\,\frac{\partial \lambda_i}{\partial x_j}+f_k\,\frac{\partial \lambda_i}{\partial x_k}
$$
on an arbitrarily chosen  curve~$r_i$,
where  $\nabla \lambda_i$ is the normal  to the surface~$\Lambda_i$ evaluated
along~$r_i$ and
${\bold V}=(f_1,f_2,f_3)$ is the vector field associated with the differential system~\eqref{three_equat}.

\begin{lemma}\label{normal_inside_gen}
For any generalized Wallach space with $a_1,a_2,a_3\in (0,1/2)$
for every $i=1,2,3$ the normal $\nabla \lambda_i$ of the surface~$\Lambda_i$
is directed inside $R$
at every point of the  curve $r_i$.
\end{lemma}

\begin{proof}
Consider    $\Lambda_i \subset \partial(R)$.
It suffices to consider points of  $r_i$.
Using Lemma~\ref{param_r_i} we obtain
$$
\frac{\partial \lambda_i}{\partial x_i}=
\frac{a_i\left(t-m_i\right)\left(t-M_i\right)}{tx_i^2}>0
$$
for all $a_1,a_2,a_3\in (0,1/2)$ and all $t\in (0,m_i)\cup (M_i, +\infty)$,
where $m_i$ and $M_i$ are roots of $t^2-a_i^{-1}t+1=0$ known from~\eqref{m_and_M}.
This means that the normal
$\nabla \lambda_i:=\left(\tfrac{\partial \lambda_i}{\partial x_1}, \tfrac{\partial \lambda_i}{\partial x_2}, \tfrac{\partial \lambda_i}{\partial x_3}\right)
$
is directed into $R$ at every point of the curve $r_i$ on the surface $\Lambda_i\subset \partial(R)$,
because of points of $R$ lie above the surface $\Lambda_i$ in the coordinate system $(i,j,k)$
according to Lemma~\ref{Delta_Conic_gen}.
Note also that
\begin{equation*}\label{gradd}
\frac{\partial \lambda_i}{\partial x_j}=
-\frac{\sqrt{\left(t-m_i\right)\left(t-M_i\right)}}{2tx_i^2}\,\, (2a_it-1), \qquad
\frac{\partial \lambda_i}{\partial x_k}=
\frac{\sqrt{\left(t-m_i\right)\left(t-M_i\right)}}{2tx_i^2}\,\, (t-2a_i).
\end{equation*}
Since the inequalities
$0<m_i<2a_i<\frac{1}{2a_i}<M_i$ hold
for all  $a_i\in (0,1/2)$ (see~\eqref{tlar}),
we have $\frac{\partial \lambda_i}{\partial x_j}>0$ and
$\frac{\partial \lambda_i}{\partial x_k}<0$ for $t\in (0,m_i)$.
Respectively, $\frac{\partial \lambda_i}{\partial x_j}<0$ and
$\frac{\partial \lambda_i}{\partial x_k}>0$
on the component~$r_{ik}$ corresponding to the interval $(M_i, +\infty)$.

Similar reasonings can be repeated for $\Lambda_j$ and $\Lambda_k$.

Everything that was established here
also holds for the corresponding \glqq working`` subintervals
$(0,t_{ij}]\subset (0,m_i)$ and
$[t_{ik}, +\infty)\subset (M_i,+\infty)$
for $t_{ij}$ and $t_{ik}$ evaluated in Lemma~\ref{pairwise_inter_r}.
Lemma~\ref{normal_inside_gen} is proved.
\end{proof}

\begin{lemma}\label{znak_inner_prod}
For  an arbitrary generalized Wallach space with $a_1,a_2,a_3\in (0,1/2)$
for every $i=1,2,3$ and for all $t\in (0,m_i) \cup (M_i,+\infty)$
the sign of the inner product $({\bold V}, \nabla \lambda_i)$
evaluated at an arbitrary point of the curve $r_i$
coincides with the sign of the polynomial
\begin{equation}\label{Hgoi}
h(t)=c_0t^4+c_1t^3+c_2t^2+c_1t+c_0
\end{equation}
with coefficients
\begin{equation}\label{alpha_coef_gen}
\begin{array}{l}
c_0:=a_i^2\left(1-2a_i+2a_j+2a_k\right)\left(2a_i+2a_j+2a_k-1\right),\\
c_1:=a_i(1-2a_i)^3-4a_i\left(4a_i^2+1\right)(a_j+a_k)^2,\\
c_2:=\left(16a_i^4+16a_i^2+1\right)(a_j+a_k)^2+2a_i(1-a_i)(1-2a_i)^2.
\end{array}
\end{equation}
\end{lemma}

\begin{proof}
As calculations show $({\bold V}, \nabla \lambda_i)$ can be represented by
the following expression:
\begin{equation*}
\begin{array}{l}
({\bold V}, \nabla \lambda_i)=f_i\,\dfrac{\partial \lambda_i}{\partial x_i} + f_j\,\dfrac{\partial \lambda_i}{\partial x_j} + f_k\,\dfrac{\partial \lambda_i}{\partial x_k} \\
= \dfrac{a_i}{2\,(a_ia_j + a_ia_k + a_ja_k)x_i^2\,x_j^2\,x_k^2}
\Big[
(a_i+a_j)(a_j+a_k)(a_k+a_i)\left(x_j^4+x_k^4-x_i^4\right)\\
+ a_k(a_i+a_j)\left(x_i^2-x_j^2\right)x_ix_j+a_j(a_i+a_k)\left(x_i^2-x_k^2\right)x_ix_k
-2a_i(a_j+a_k)\left(x_j^2+x_k^2\right)x_jx_k\\
+\big\{\left(2a_ia_k+a_ia_j+a_ja_k-a_j\right)x_j+\left(2a_ia_j+a_ia_k+a_ja_k-a_k\right)x_k\big\}\,x_ix_jx_k\\
+ (a_j+a_k)\left(2a_i^2-2a_ia_j-2a_ia_k-2a_j a_k+1\right)x_j^2x_k^2\Big].
\end{array}
\end{equation*}
Substituting  the parametric equations
$x_i=\phi_i(t)$, $x_j=\phi_j(t)$, $x_k=\phi_k(t)$,  $t\in (0,m_i)\cup (M_i, +\infty)$,
of the curve~$r_i$  found in Lemma~\ref{param_r_i}
into the latter expression for $({\bold V}, \nabla \lambda_i)$ we obtain
$
\left({\bold V}, \nabla \lambda_i\right)=\dfrac{1}{2a_itx_i^2}\, \left(F-G\right)
$,
where
$F := (a_j+a_k)(t-2a_i)\left(2a_it-1\right)$ and
$G := (1-2a_i)\, a_i\, (t+1)\sqrt{t^2-a_i^{-1}t+1}$.
The inequalities~\eqref{tlar} imply $F>0$ and  $G>0$
for all  $t\in (0,m_i)\cup (M_i, +\infty)$ and all $a_i, a_j, a_k\in (0,1/2)$.
Therefore for $t\in (0,m_i)\cup (M_i, +\infty)$
the sign of $({\bold V}, \nabla \lambda_i)$
coincides with the sign of the following fourth degree polynomial
in $t$:
$$
h(t):=F^2-G^2=(a_j+a_k)^2(t-2a_i)^2\left(2a_it-1\right)^2 -
(1-2a_i)^2a_i(t+1)^2\left(a_it^2-t+a_i\right).
$$
Collecting its similar terms
we obtain the polynomial~\eqref{Hgoi} with coefficients in~\eqref{alpha_coef_gen}.
Lemma~\ref{znak_inner_prod} is proved.
\end{proof}

\medskip

Since  $t=0$ is not a root of~$h(t)=0$
introducing a new variable $y:=t+t^{-1}$ the equation $h(t)=0$
can be reduced to a quadratic equation
\begin{equation}\label{P2(y)}
p(y):=c_0 y^2+c_1 y+(c_2-2c_0)=0.
\end{equation}

\begin{lemma}\label{lem_20_07_2024}
Two different real roots
$t_1=\dfrac{y_0-\sqrt{y_0^2-4}}{2}\in (0,m_i)$ and
$t_2=t_1^{-1}=\dfrac{y_0+\sqrt{y_0^2-4}}{2}\in (M_i, +\infty)$
of the equation $h(t)=0$ correspond to a given root~$y_0\in \left(a_i^{-1}, +\infty\right)$
of the equation $p(y)=0$.
\end{lemma}

\begin{proof}
The proof follows from the properties of the functions
$x \mapsto \dfrac{1-\sqrt{1-4x^2}}{2x}$ (increases) and
$x \mapsto \dfrac{1+\sqrt{1-4x^2}}{2x}$ (decreases) for $0<x<1/2$.
\end{proof}

\begin{lemma}\label{znaki_coeff}
In an arbitrary generalized Wallach space with $a_1,a_2,a_3\in (0,1/2)$
the following assertions hold for the coefficients $c_n=c_n(a_1,a_2,a_3)$, ($n=0,1,2$), of the polynomial $p(y)$ in~\eqref{P2(y)}:

$c_2>0$ and  $c_2-2c_0>0$ for all $a_1,a_2,a_3\in (0,1/2)$;

$c_0=0$ and $c_1<0$  if   $a_1+a_2+a_3=1/2$;

$c_0<0$ if  $a_1+a_2+a_3<1/2$;

$c_0>0$ and $c_1<0$  if  $a_1+a_2+a_3>1/2$.
\end{lemma}

\begin{proof}
The inequality $c_2>0$ follows from $16a_i^4+16a_i^2+1>0$ (see~\eqref{alpha_coef_gen}).
Introducing a new variable  $\mu:= a_j+a_k$ we obtain the inequality
 $$
 c_2-2c_0=\left(1+4a_i^2\right)\mu^2+2a_i(1-2a_i)^2>0
 $$
true for all $a_i,a_j,a_k\in (0,1/2)$.

\medskip
{\it The case $a_1+a_2+a_3=1/2$}. Replacing  $a_j+a_k$ by $1/2-a_i$ in~\eqref{alpha_coef_gen},
we obtain
$c_0=0$ and $c_1=-2(1+2a_i)(1-2a_i)^2a_i^2<0$.

\medskip
{\it The case $a_1+a_2+a_3<1/2$}. Clearly $c_0<0$.

\medskip
{\it The case $a_1+a_2+a_3>1/2$}. Obviously  $c_0>0$.
The coefficient $c_1$  can be considered as a function
$$
c_1=W(\mu):=a_i(1-2a_i)^3-4a_i\left(1+4a_i^2\right)\mu^2
$$
of the independent variable $\mu= a_j+a_k$.
Clearly $W(\mu)$ admits two zeros $\mu=-\mu_0$ and $\mu=\mu_0$, where
$\mu_0:=\dfrac{1-2a_i}{2}\sqrt{\dfrac{1-2a_i}{1+4a_i^2}}$.
Since  $0<\mu_0<\dfrac{1}{2}-a_i$
we have $W(\mu)\Big|_{\mu=1/2-a_i}<0$.
Then  $W(\mu)<0$ for all  $\mu>1/2-a_i$ meaning that
$c_1<0$ at $a_i+a_j+a_k>1/2$. Lemma~\ref{znaki_coeff} is proved.
\end{proof}

\section{Proofs of Theorems~\ref{thm_sum_a_i<1/2}~-- \ref{thm_3} and Corollary~\ref{corol_1}}

It suffices to consider the system~\eqref{three_equat} on its invariant surface~$\Sigma$
defined by the equation  $\operatorname{Vol}:=x_1^{1/a_1}x_2^{1/a_2}x_3^{1/a_3}=1$.
Hence we are interested in parameters of metrics which belong to the domain $\Sigma \cap R$ bounded by
the curves $r_i=\Sigma \cap \Lambda_i$ for $i=1,2,3$.
In order to establish behavior of trajectories of the system~\eqref{three_equat}
with respect to that domain  we will study their interrelations (namely, the number of intersection points) with curves forming the boundary of that domain.

\subsection{The case $a_1+a_2+a_3\le 1/2$}
Examples~\ref{ex_sum_a_i<1/2} and \ref{ex_sum_a_i=1/2}
illustrate Theorem~\ref{thm_sum_a_i<1/2}.

\bigskip

\begin{proof}[Proof of Theorem~\ref{thm_sum_a_i<1/2}]
It suffices to consider the border curve $r_i$ only.
Similar reasoning can be repeated for the curves $r_j$ and $r_k$
by simple permutations of indices $\{i,j,k\}=\{1,2,3\}$ in~\eqref{parcurvegen}.
Let us consider the cases  $a_1+a_2+a_3<1/2$  and  $a_1+a_2+a_3=1/2$ separately.

\smallskip

{\bf The case $a_1+a_2+a_3<1/2$}.
For $a_1+a_2+a_3<1/2$  the discriminant
\begin{equation*}\label{disP2(y)}
c_1^2+8c_0^2-4c_0c_2
\end{equation*}
of the quadratic equation $p(y)=c_0 y^2+c_1 y+(c_2-2c_0)=0$ is positive,
because $c_0<0$ and $c_2>0$ by Lemma~\ref{znaki_coeff}.
Hence  $p(y)=0$ admits two different real roots.
We claim that actually a single root of $p(y)=0$ can belong to $(a_i^{-1}, +\infty)$.
Indeed
\begin{equation}\label{sign_P2ofa_inverse}
p(a_i^{-1})=(a_k+a_j)^2(1-2a_i)^2(1+2a_i)^2>0
\end{equation}
independently on $a_i,a_j,a_k$.
It is clear also
$p(+\infty):=\lim_{y \to +\infty} p(y)=-\infty$
due to  $c_0<0$.
Therefore $p(y_0)=0$ for some unique $y_0\in (a_i^{-1}, +\infty)$.

No root of the quadratic polynomial $p(y)$
can be contained in the interval $[0,a_i^{-1}]$
since  $p(0)=c_2-2c_0>0$ by the same Lemma~\ref{znaki_coeff}.
Then another root of $p(y)=0$ distinct from $y_0$ must be negative
and is not of interest providing  negative roots to $h(t)=0$.
Therefore the polynomial $h(t)=c_0t^4+c_1t^3+c_2t^2+c_1t+c_0$ of degree~$4$ admits two
positive roots $t_1\in (0,m_i)$ and $t_2\in (M_i, +\infty)$
which correspond to $y_0$ according to Lemma~\ref{lem_20_07_2024}.
Note also that $h(t)$ has a \glqq $\cap$``-shaped graph for $t>0$ and
$h(t)<0$ for $t\in (0,t_1)\cup (t_2,+\infty)$
and $h(t)>0$ for  $t\in (t_1, m_i)\cup (M_i,t_2)$.
Then according to Lemma~\ref{znak_inner_prod} we have
$$
\operatorname{sign}({\bold V}, \nabla \lambda_i)=
\begin{cases}
-1, &\mbox{if} \quad t\in (0,t_1),\\
~~0, &\mbox{if} \quad t=t_1,\\
+1,  &\mbox{if} \quad t\in (t_1, m_i)
\end{cases}
\mbox{~~~ and ~~~}
\operatorname{sign}({\bold V}, \nabla \lambda_i)=
\begin{cases}
+1,  &\mbox{if} \quad t\in (M_i, t_2), \\
~~0, &\mbox{if} \quad t=t_2,\\
-1, &\mbox{if} \quad t\in  (t_2, +\infty)
\end{cases}
$$
respectively on~$r_{ij}$ and~$r_{ik}$.
To reach  our aims it is quite enough to show that
trajectories of the system~\eqref{three_equat}
will leave the domain~$R$
across some part of some component of the curve $r_i$.
Take, for example, the second component $r_{ik}$ of~$r_i$
(a similar analysis can be given for the component
$r_{ij}$ of~$r_i$).

As it follows from the formulas above
$({\bold V}, \nabla \lambda_i)>0$ for  $t\in (M_i,t_2)$.
This means that  on every point of a part of the curve  $r_{ik}$
which corresponds to $(M_i,t_2)$ under the parametrization~\eqref{parcurvegen}
the associated vector field ${\bold V}$ of the system~\eqref{three_equat}
forms  an acute angle with the normal  $\nabla \lambda_i$
of the surface~$\Lambda_i\subset \partial(R)$
(in fact its component $\Lambda_{ik}$).
Since  the normal $\nabla \lambda_i$ is directed inside~$R$ at every point of  $r_i$
due to Lemma~\ref{normal_inside_gen}, we conclude that
trajectories of~\eqref{three_equat} attend~$R$ at $t\in (M_i,t_2)$.

For $t\in (t_2, +\infty)$  clearly
$({\bold V}, \nabla \lambda_i)<0$ meaning
that~${\bold V}$ forms  an obtuse angle with the normal~$\nabla \lambda_i$.
Thus we established that trajectories of the system~\eqref{three_equat}
originated in~$R$ will leave $R$ at least through a part of the curve~$r_{ik}$ and never cross or touch $r_{ik}$ again for all $t>t_2$.
 This means that there exist metrics that lose the positivity of the Ricci curvature  in finite time.
Theorem~\ref{thm_sum_a_i<1/2} is proved for the case  $a_1+a_2+a_3<1/2$.

\smallskip

{\bf The case $a_1+a_2+a_3=1/2$}.
Obviously we deal with the linear function  $p(y)=c_1 y+c_2$ since
$c_0=0$, $c_1<0$ and $c_2>0$ by Lemma~\ref{znaki_coeff}.
Moreover, the only root $y_0$ of  $p(y)=0$ is positive.
Since $p(a_i^{-1})>0$ due to~\eqref{sign_P2ofa_inverse}
and
$$
p(+\infty):=\lim_{y \to +\infty} p(y)=-\infty
$$
due to  $c_1<0$ we have $y_0\in (a_i^{-1}, +\infty)$.
Lemma~\ref{lem_20_07_2024} implies that the polynomial $h(t)=(c_1t^2+c_2t+c_1)\,t$ of degree $3$ has two roots $t_1\in (0,m_i)$ and  $t_2\in (M_i, +\infty)$ which correspond to $y_0$.
We are in the same situation as in the considered case $a_1+a_2+a_3<1/2$.
Further reasoning is similar.
Theorem~\ref{thm_sum_a_i<1/2} is proved.
\end{proof}

\begin{remark}\label{rem_240724}
According to Remark~\ref{sec_tails} we had to consider
 intervals $(0, t_{ij}]\subset  (0,m_i)$
and $[t_{ik}, +\infty)\subset(M_i,+\infty)$ in Theorem~\ref{thm_sum_a_i<1/2}
instead of~$(0, m_i)$ and $(M_i,+\infty)$
cutting of the extra pieces of the curves $r_i$, $r_j$ and $r_k$
that remain after their pairwise intersections.
Although such a replacing   does not cancel the existence
of trajectories leaving~$R$ in any way, and therefore does not  affect
the proof of Theorem~\ref{thm_sum_a_i<1/2},
we get a little clarification to the qualitative picture of
incoming trajectories.
Such a picture depends on the relative position of $t_1$ and $t_{ij}$ within the interval $(0,m_i)$, for instance, if the component $r_{ij}$ is considered.

If  $t_{ij}< t_1$ then $h(t)<0$ on all \glqq useful``
part of $r_{ij}$ corresponding to~$t\in (0,t_{ij}]$, hence $h(t)$ changes its sign only at the
 \glqq tail`` of the curve $r_{ij}$ behind the point $P_{ij}$.
Therefore $({\bold V}, \nabla \lambda_i)<0$ for all  $t\in (0,t_{ij}]$
and hence every point on~$r_{ij}$
emits trajectories towards the exterior of~$R$. The interval  $(0,t_{ij}]$ above
should be replaced by $(0,t_{ij})$ if $t_{ij}=t_1$.
In the case $t_{ij}>t_1$ trajectories
come into~$R$ through the part of  $r_{ij}$
corresponding to~$t\in (t_1,t_{ij}]\subset (t_1, m_i)$
and leave  $R$ through its  part corresponding to
$t\in (0, t_1)$.

Analogously, if $t_{ik}<t_2$ then trajectories attend~$R$ for
 $t\in [t_{ik}, t_2)\subset (M_i,t_2)$ and  leave~$R$ for
$t>t_2$.
If $t_{ik}>t_2$ (respectively $t_{ik}=t_2$) then outgoing of trajectories from~$R$
happens through every point of~$r_{ik}$ for $t\ge t_{ik}$
(respectively $t> t_{ik}$).

Thus analyzing all possible outcomes we conclude
that trajectories of~\eqref{three_equat}
will leave~$R$ in any case whenever $a_1+a_2+a_3\le 1/2$:
either through entire curve $r_i$ either through some its part.
\end{remark}

\begin{remark}\label{rem_280724}
In the case $a_1+a_2+a_3\le 1/2$
no metric~\eqref{metric}  can evolve again into a metric with positive Ricci curvature,
being transformed once into a metric with mixed Ricci curvature,
but some metrics with mixed Ricci curvature can be transformed into metrics with positive Ricci curvature, and then again into metrics with mixed Ricci curvature.
This follows  from the proof of Theorem~\ref{thm_sum_a_i<1/2} according to which
no trajectory of~\eqref{three_equat}
returns back to $R$ leaving~$R$ once,
but some trajectories attend $R$ and come back if initiated in the exterior of~$R$.
\end{remark}

\medskip

\subsection{The case $a_1+a_2+a_3> 1/2$}
Recall the specific parameters $\theta_i:=a_i-\frac{1}{2}+\frac{1}{2}\sqrt{\frac{1-2a_i}{1+2a_i}}$
exposed in the text of Theorem~\ref{thm_sum_a_i>1/2}.
Clearly $\theta_i\in (0,1)$  for all $i=1,2,3$ and all $a_i\in (0,1/2)$.

\bigskip

\begin{proof}[Proof of Theorem~\ref{thm_sum_a_i>1/2}]
It suffices to consider the curve $r_i$ only.
In formulas~\eqref{alpha_coef_gen} for $c_0, c_1$ and $c_2$
replace the sum $a_j+a_k$  by $\frac{1}{2}-a_i+\theta$:
\begin{equation}\label{alpha_coef_gen1}
\begin{array}{l}
c_0=4a_i^2\,\theta\left(1-2a_i+\theta\right),\\
c_1:=a_i\left(1-2a_i\right)^3-a_i\left(4a_i^2+1\right)\left(1-2a_i+2\theta\right)^2,\\
c_2:=\dfrac{1}{4}\left(16a_i^4+16a_i^2+1\right)\left(1-2a_i+2\theta\right)^2+
2a_i(1-a_i)(1-2a_i)^2.
\end{array}
\end{equation}
Since each  coefficient $c_0,c_1$ and $c_2$ depends on the parameters $a_i$ and $\theta$ only
so does the discriminant  of the quadratic equation
$p(y)=0$ in~\eqref{P2(y)}.
By this reason denote them by $D_i$:
\begin{multline*}\label{expr_15062024}
D_i:=c_1^2-4c_0(c_2-2c_0)=c_1^2+8c_0^2-4c_0c_2\\
=-4a_i^2(1+2a_i)^2(1-2a_i)^3
\Big\{(1+2a_i)\theta^2+\left(1-4a_i^2\right)\theta-(1-2a_i)a_i^2\Big\}.
\end{multline*}

Since $c_0>0$ for $a_i+a_j+a_k>1/2$ and $c_2>0$ by Lemma~\ref{znaki_coeff},
the polynomial  $D_i$ can accept  values of any sign.
Firstly, we find roots of $D_i=0$.
Observe that the value $\theta=\theta_i$
is exactly the unique root of the quadratic equation
$$
(1+2a_i)\theta^2+\left(1-4a_i^2\right)\theta-(1-2a_i)a_i^2= 0.
$$
Its another root is negative.

\medskip
$(1)$
{\it The case $D_i<0$}. The inequality $D_i<0$ is equivalent to
$$
(1+2a_i)\theta^2+\left(1-4a_i^2\right)\theta-(1-2a_i)a_i^2>0
$$
which has solutions $\theta>\theta_i$.
Therefore the equation $p(y)=0$, and hence the equation $h(t)=0$
has no real roots if $\theta\in \left(\theta_i,1\right)$.
Then  $h(t)>0$ for all $t\in (0,m_i)\cup (M_i,+\infty)$ due to $c_0>0$,
in other words, by Lemma~\ref{znak_inner_prod} the sign of
$({\bold V}, \nabla \lambda_i)$ is positive along the curve~$r_i$,
that means there is no arc or even a point on the curve~$r_i$,
across which trajectories of the system~\eqref{three_equat} could leave the domain~$R$,
because of  the normal $\nabla \lambda_i$ of the
cone~$\Lambda_i$ is directed inside~$R$ at every point of~$r_i$
independently on values of $a_1$, $a_2$ and $a_3$ due to Lemma~\ref{normal_inside_gen}.

\medskip

{\it The case $D_i= 0$}. Since  $D_i=0$ happens at $\theta=\theta_i$ only,
for the sum $a_j+a_k$ we have
$$
a_j+a_k=\frac{1}{2}-a_i+\theta_i=\frac{1}{2}\sqrt{\dfrac{1-2a_i}{1+2a_i}}.
$$
The coefficients~\eqref{alpha_coef_gen1} take the following forms at $\theta=\theta_i$:
{\renewcommand{\arraystretch}{2.5}
\begin{equation*}\label{alpha_coef_gen2}
\begin{array}{l}
\widetilde{c}_0:=-a_i^2\left((1-2a_i)^2-4(a_j+a_k)^2\right)
=4a_i^4 \cdot \dfrac{1-2a_i}{1+2a_i},\\
\widetilde{c}_1:=a_i\left(1-2a_i\right)^3-a_i\left(4a_i^2+1\right)\cdot \dfrac{1-2a_i}{1+2a_i},\\
\widetilde{c}_2:=\dfrac{1}{4}\left(16a_i^4+16a_i^2+1\right)
\cdot \dfrac{1-2a_i}{1+2a_i}+2a_i\left(1-a_i\right)\left(1-2a_i\right)^2.
\end{array}
\end{equation*}
}

In what follows an estimation
\begin{equation*}\label{alpha of teta0}
y_0-a_i^{-1}=-\frac{\widetilde{c}_1}{2\widetilde{c}_0}-a_i^{-1}
=\frac{(1-2a_i)^2(1+2a_i)^2}{4a_i^2(1-4a_i^2)}=
\frac{1-4a_i^2}{4a_i^2}>0
\end{equation*}
 for the unique root
$y_0=-\dfrac{\widetilde{c}_1}{2\widetilde{c}_0}$ of the equation
$\widetilde{P}_2(y)=\widetilde{c}_0y^2+\widetilde{c}_1y+
\widetilde{c}_2-2\widetilde{c}_0=0$.

Since  $y_0\in (a_i^{-1}, +\infty)$ is the single root  of $p(y)$ of multiplicity $2$
the polynomial $h(t)$ of degree~$4$ has two zeros $t_1$  and $t_2$ each of multiplicity $2$
such that  $t_1\in (0,m_i)$ and $t_2\in (M_i,+\infty)$.
Since $c_0>0$  then $h(t)>0$ anywhere on $(0,m_i)\cup (M_i,+\infty)$ excepting
the points $t_1$ and $t_2$ (the shape of the graph of $h(t)$ is similar to the letter \glqq W`` touching the $t$-axis at the points $t_1$ and $t_2$).
Then by Lemma~\ref{znak_inner_prod} no trajectory of~\eqref{three_equat}
can leave~$R$ through its boundary curve~$r_i$.
Note that in the case $D_i=0$ we are able to find the mentioned   multiple roots of $h(t)$ explicitly:
$$
t_{1,2}=\frac{1}{8a_i^2}
\left[-4a_i^2+4a_i+1\mp(1+2a_i)\sqrt{(1+6a_i)(1-2a_i)}\right], \\
$$
since $h(t)=(1-2a_i)(1+2a_i)^{-1}\Big(a_i^2t^2+\left(a_i^2-a_i-0.25\right)t+a_i^2\Big)^2$.

Thus uniting the cases  $\theta>\theta_i$ ($D_i<0$) and  $\theta=\theta_i$ ($D_i=0$)
we conclude that no trajectory of~\eqref{three_equat} can leave $R$ through $r_i$.
The same will hold on the curves  $r_j$ and $r_k$, if
$\theta\ge \max\left\{\theta_i, \theta_j, \theta_k\right\}$
(all metrics maintain the positivity of the Ricci curvature).
This finishes the proof of the first assertion in Theorem~\ref{thm_sum_a_i>1/2}.
See also Example~\ref{ex_sum>1/2_4}.

\medskip
$(2)$
To prove the second assertion in Theorem~\ref{thm_sum_a_i>1/2}
assume that $\theta< \theta_i$ for some $i$.
Such an assumption equivalent to  $D_i> 0$.
We claim  that  both roots of $p(y)=0$
belong to $(a_i^{-1}, +\infty)$ in this case.
We are in a situation where a quadratic function with
a positive discriminant takes positive values at the ends of the interval:
$$
p(a_i^{-1})>0, \qquad   p(+\infty):=\lim_{y \to +\infty} p(y)=+\infty
$$
due to  $c_0>0$.
Therefore, the analysis used in the proof of Theorem~\ref{thm_sum_a_i<1/2}
is not applicable here: such an interval can contain either both roots of the equation
or none.

\smallskip

{\it  No root can belong to $(-\infty, 0)$}.
Suppose by contrary  $p(y)=0$ has roots $y_1,y_2\in (-\infty, 0)$.
Then $-\dfrac{c_1}{2c_0}=\dfrac{y_1+y_2}{2}<0$ despite the facts
$c_0>0$ and $c_1<0$  established
in Lemma~\ref{znaki_coeff} in the case $a_1+a_2+a_3>1/2$.

\smallskip

{\it  No root can belong to $(0, a_i^{-1})$}.
There is a similar situation at the ends of the interval:
$p(0)=c_2-2c_0>0$ by Lemma~\ref{znaki_coeff}
and $p(a_i^{-1})>0$ by~\eqref{sign_P2ofa_inverse}.
Introduce a function
$$
U(\mu):=-\dfrac{c_1}{2c_0}-a_i^{-1},
$$
where the sum $a_j+a_k$ is replaced by  $\mu=a_j+a_k$ in  $c_0$ and $c_1$.
Then $U(\mu)>0$ is equivalent to the inequality
$4\mu^2-(1-2a_i)<0$
admitting solutions $\mu\in \left(0,\overline{\mu}\right)$,
where $\overline{\mu}:=\dfrac{\sqrt{1-2a_i}}{2}$ is the positive root of~$U(\mu)=0$.
Since $\mu=1/2-a_i+\theta$ the following value of $\theta$
corresponds to $\overline{\mu}$:
$
\overline{\theta}_i=\dfrac{2a_i-1+\sqrt{1-2a_i}}{2}
$.
It follows then $U>0$ for $\theta\in \left(0,\overline{\theta}_i\right)$
and $U<0$  for $\theta\in \left(\overline{\theta}_i, 1\right)$.
It is not difficult to see that
$$
\overline{\theta}_i=
\dfrac{4a_i^2-1+\sqrt{\left(1-4a_i^2\right)(1+2a_i)}}{2(1+2a_i)}>
\dfrac{4a_i^2-1+\sqrt{1-4a_i^2}}{2(1+2a_i)}=\theta_i.
$$
Therefore, $U=-\dfrac{c_1}{2c_0}-a_i^{-1}>0$ for values
$\theta\in \left(0, \theta_i\right)\subset \left(0,\overline{\theta}_i\right)$,
which correspond to the case $D_i>0$.
It follows then $p(y)=0$ has no roots in $(0, a_i^{-1})$.
Supposing the contrary we would conclude $-\dfrac{c_1}{2c_0}=\dfrac{y_1+y_2}{2}<a_i^{-1}$ from
  $0<y_1<a_i^{-1}$ and $0<y_2<a_i^{-1}$ despite $U>0$.

\smallskip

Thus we showed that {\it both roots of $p(y)=0$
belong to}\/ $(a_i^{-1}, +\infty)$ in the case $\theta<\theta_i$.
Denote them by  $y_1$ and~$y_2$. Without loss of generality assume that  $y_1<y_2$.
By Lemma~\ref{lem_20_07_2024}
two different roots $t_1\in (0, m_i)$ and $t_2=t_1^{-1}\in (M_i, +\infty)$ of $h(t)=0$
correspond to $y_1$.
Analogously there are roots $t_3\in (0, m_i)$ and $t_4=t_3^{-1}\in (M_i, +\infty)$ of
$h(t)=0$ corresponding to $y_2$.
It is easy to see that
$$
t_3<t_1<m_i<M_i<t_2<t_4.
$$

Thus each of the intervals $(0,m_i)$ and $(M_i,+\infty)$ contains
exactly two roots of the polynomial~$h(t)$ which has a \glqq W``-shaped graph.
It follows then exactly
two points are expected  on each component of the curve $r_i$  at which
the vector field  ${\bold V}=(f_1,f_2,f_3)$ changes its direction
in relation to the domain $R$ (\glqq into $R$`` or \glqq from $R$``).
Consider the component~$r_{ik}$ only with $t\in (M_i, +\infty)$.
Then $M_i<t_2<t_4$ implies that
trajectories of~\eqref{three_equat} come into~$R$ through~$r_{ik}$
and never cross or touch $r_{ik}$ again  for all~$t>t_4$.
For the component $r_{ij}$ with  $t\in (0,m_i)$
a similar analysis can be carried out
by mirroring of $M_i<t_2<t_4$.

Thus in the case~$\theta<\theta_i$
trajectories of~\eqref{three_equat} can intersect the border curve~$r_i$ twice according to the scenario
\glqq towards~$R$~-- from~$R$~-- towards~$R$ again`` (see also Remark~\ref{rem_270824}).
A similar scenario will take place on  the border curves~$r_j$ or (and) $r_k$,
if $\theta<\theta_j$ or (and) $\theta< \theta_k$
returning trajectories to~$R$ even they ever left $R$ once.
Clearly if  $\theta\ge \theta_j$ or $\theta\ge \theta_k$ at $\theta<\theta_i$
then the curve $r_j$ or $r_k$ is locked for trajectories to leave $R$ as proved before.
After such an analysis we conclude that at least some metrics with positive Ricci curvature can preserve the positivity of the Ricci curvature. The second assertion in Theorem~\ref{thm_sum_a_i>1/2} is also proved.
Theorem~\ref{thm_sum_a_i>1/2} is proved.
\end{proof}

\begin{remark}\label{rem_270824}
Let us  comment the behavior of trajectories in more detail
 in the case $\theta<\theta_i$ in Theorem~\ref{thm_sum_a_i>1/2}. Consider the example of~$r_{ik}$.
As it was  noted in Remark~\ref{rem_240724}
we are actually interested in the \glqq useful`` part of the component $r_{ik}$
parameterized by $t\in [t_{ik}, +\infty)\subset (M_i,+\infty)$,
where $t_{ik}>M_i$ by Lemma~\ref{pairwise_inter_r}.
The dynamics of trajectories of~\eqref{three_equat} depends on which link of the chain
$M_i<t_2< t_4<+\infty$ the value  $t_{ik}$ will be situated.
In the case $\theta<\theta_i$ the following
situations may happen on the component $r_{ik}$ of~$r_i$:

{\it Case 1.  $t_{ik}\in [t_4, +\infty)$}.   Trajectories move towards~$R$ through
a part of  $r_{ik}$ which corresponds to $t\ge t_{ik}$
(respectively~$t> t_{ik}$ if $t_{ik}=t_4$).
No trajectory can leave $R$.

\smallskip
{\it Case  2.  $t_{ik}\in (t_2, t_4)$}.
Trajectories leave  $R$  for $t\in \left[t_{ik}, t_4\right)$ and
return back to  $R$ for $t>t_4$.

\smallskip
{\it Case 3.  $t_{ik}\in (M_i, t_2]$}.
Trajectories come into~$R$
for $t\in \left[t_{ik},t_2\right)$,
leave~$R$ for $t\in \left(t_2, t_4\right)$ and return back to~$R$
for~$t>t_4$.

The analysis of all possible outcomes shows that
even if the strong {\it Case}\/ 1  never can occur at $\theta<\theta_i$,
some trajectories of~\eqref{three_equat} must return back to $R$ and remain there
 for all $t>t_4$ in any case.
\end{remark}

\begin{remark}\label{rem_060824}
In the case $a_1+a_2+a_3> 1/2$
some metrics with mixed Ricci curvature can be transformed into metrics with $\operatorname{Ric}>0$, after then lose $\operatorname{Ric}>0$ and restore
$\operatorname{Ric}>0$ again.
This follows from the proof of Theorem~\ref{thm_sum_a_i>1/2} according to which
trajectories of~\eqref{three_equat} can attend~$R$ twice and then remain in $R$ forever.
\end{remark}

\begin{figure}[h!]
\centering
\includegraphics[width=0.9\textwidth]{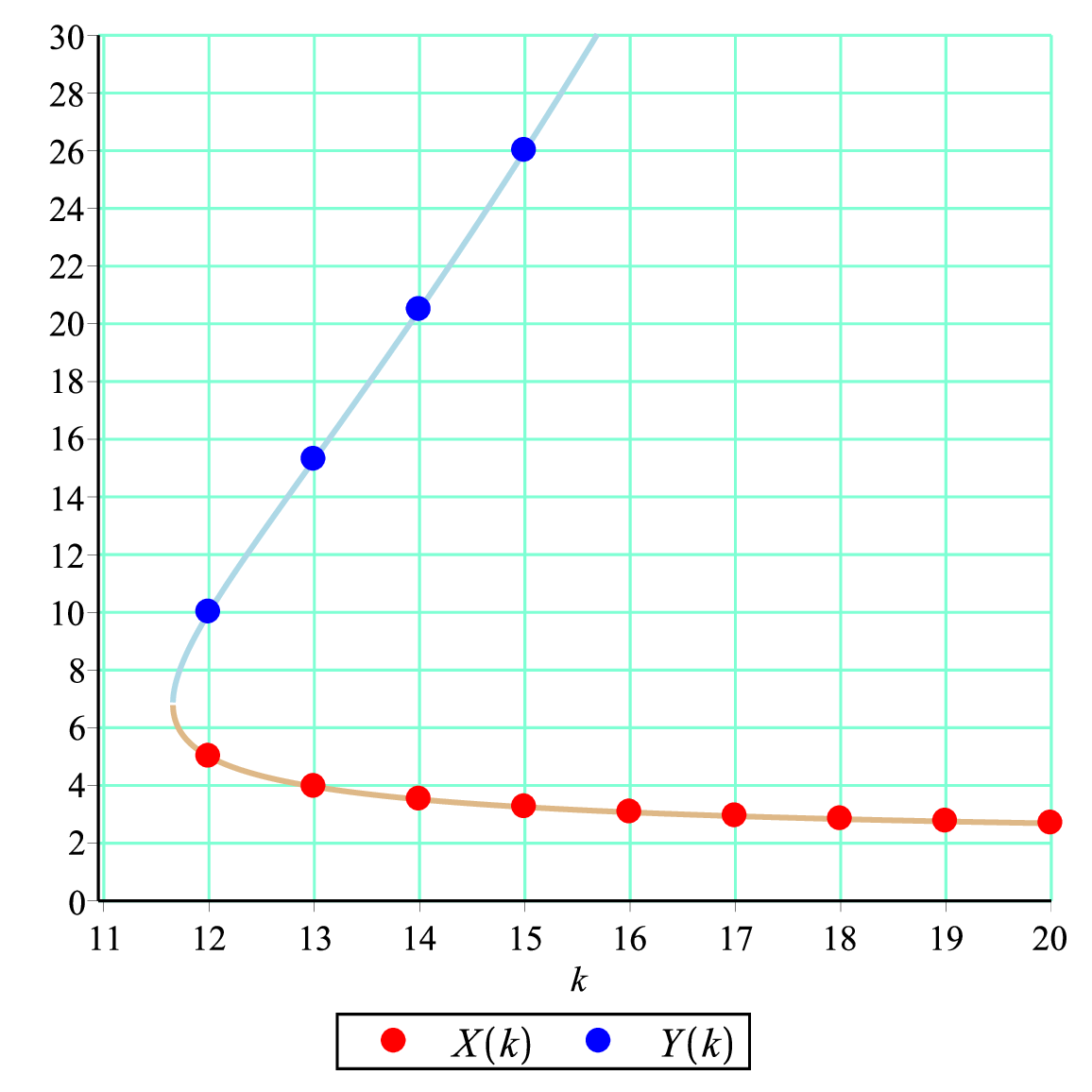}
\caption{The functions $X(k)$ and $Y(k)$}
\label{pic2}
\end{figure}

\begin{figure}[h!]
\centering
\includegraphics[width=0.9\textwidth]{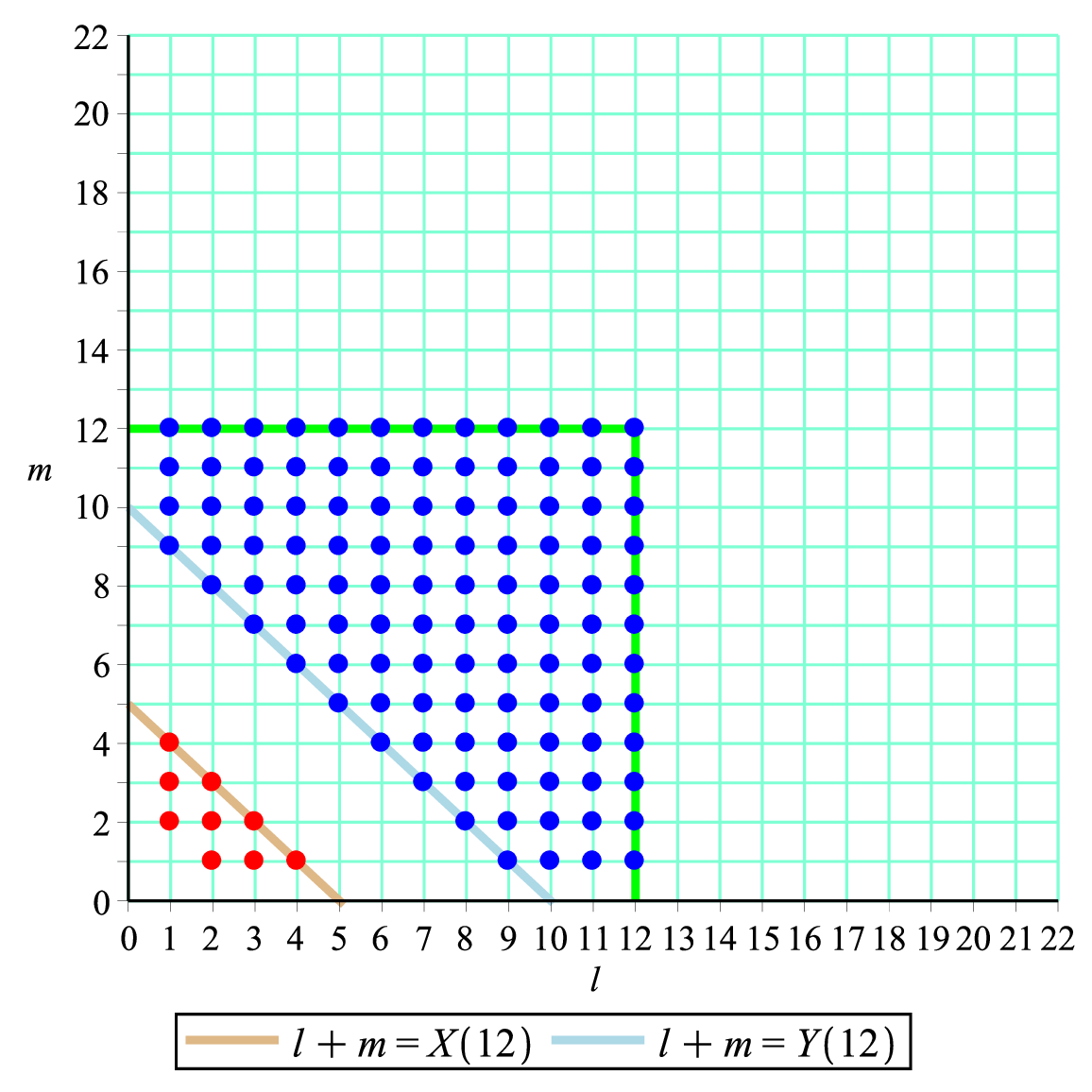}
\caption{All pairs $(l,m)$ satisfying one of the inequalities $2<l+m\le X(12)$ (red points)
or $l+m\ge Y(12)$ (blue points)}
\label{XY12}
\end{figure}

\begin{figure}[h!]
\centering
\includegraphics[width=0.9\textwidth]{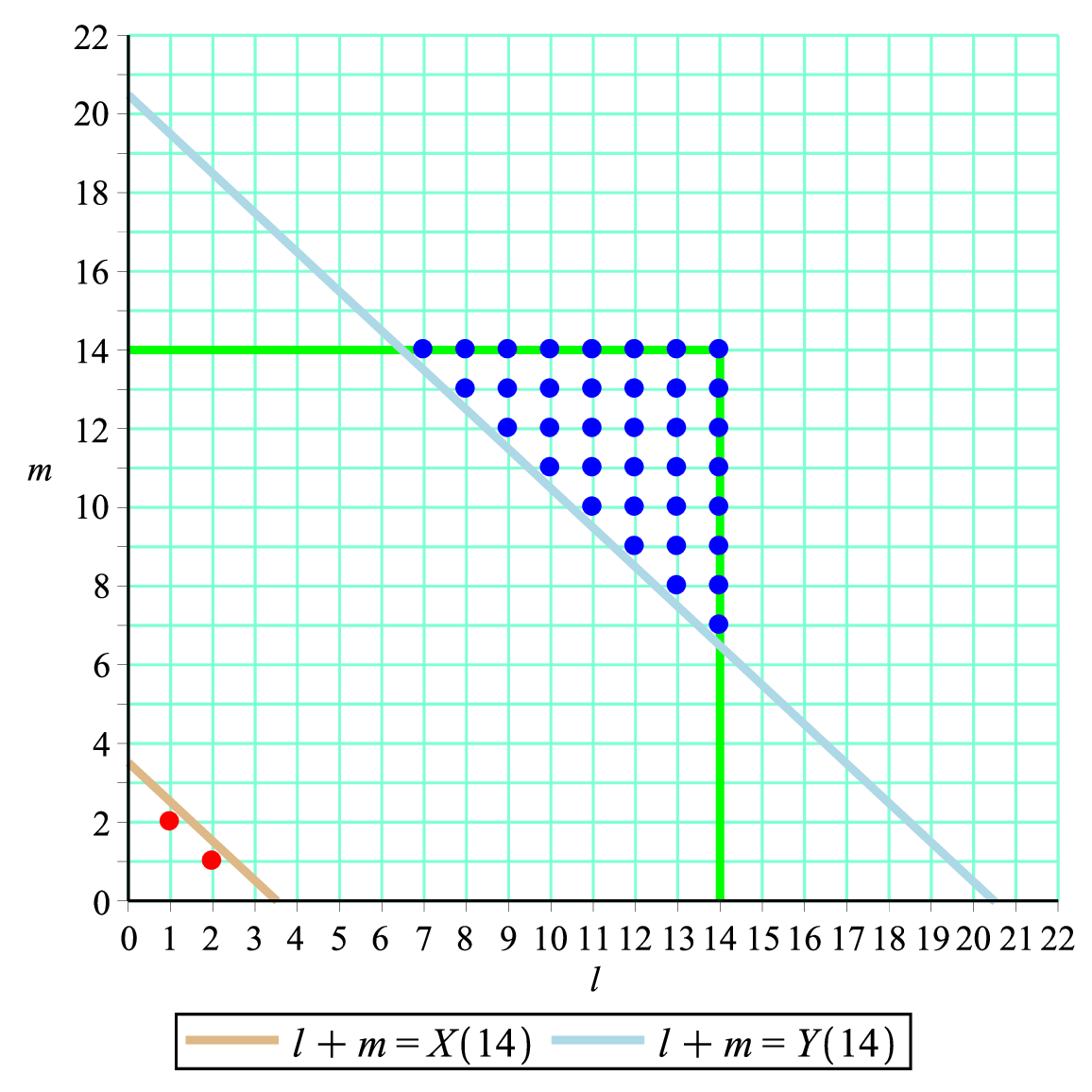}
\caption{All pairs $(l,m)$ satisfying one of the inequalities $2<l+m\le X(14)$ (red points)
or $l+m\ge Y(14)$ (blue points)}
\label{XY14}
\end{figure}

\bigskip

\begin{proof}[Proof of Theorem~\ref{thm_3}] Clearly   $\theta=a_1+a_2+a_3-\dfrac{1}{2}=\dfrac{1}{k+l+m-2}>0$ for  GWS {\bf 1}.

\smallskip

$(1)$
{\bf Case A}. {\it The condition $\max\{k,l,m\}\le 11$ holds}.
Without loss of generality consider the difference $\theta-\theta_1$ only.
Observe that $a_1=k\theta/2$ and hence
$$
\theta-\theta_1=\theta-a_1+\dfrac{1}{2}-\dfrac{1}{2}\sqrt{\dfrac{1-2a_1}{1+2a_1}}=
\theta-\frac{k\theta}{2}+\dfrac{1}{2}-\dfrac{1}{2}\sqrt{\dfrac{1-k\theta}{1+k\theta}}.
$$
It follows then  $1-k\theta\ge 0$ necessarily to be well defined.
Therefore  permissible values of  $\theta\in (0,1)$
must belong to a narrower interval  $(0, 1/k]\subseteq (0,1)$ for each natural number~$k$.
Under  such restrictions the inequality  $\theta-\theta_1\ge 0$ is equivalent to
the following system of inequalities
\begin{eqnarray*}
T(\theta):=k(k-2)^2\theta^2-(k^2-4)\theta+4\ge 0 \mbox{~~ and ~~}
0<\theta \le 1/k.
\end{eqnarray*}

Obviously $T(\theta)=4>0$ is trivial at $k=2$.
Assume that $k\ne 2$.
The discriminant
$$
D(k):=(k^2-12k+4)(k-2)^2=\left(k-\big(6-4\sqrt{2}\big)\right)
(k-2)^2\left(k-\big(6+4\sqrt{2}\big)\right)
$$
of the polynomial $T(\theta)$, where $6-4\sqrt{2}\approx 0.3431$ and $6+4\sqrt{2}\approx 11.6569$,
is negative for all $k\in \{1, 3, 4, \dots, 11\}$,
and hence  for $k\in \{1,2, \dots, 11\}$
we have $T(\theta)>0$ that means  $\theta>\theta_1$.
Then by Theorem~\ref{thm_sum_a_i>1/2}
no point on the curve~$r_1$ can emit a trajectory towards the exterior of $R$.
The same will hold for the curves $r_2$ and~$r_3$ if $\max\{k,l,m\}\le 11$
that imply $\theta\ge \max\{\theta_1, \theta_2, \theta_3\}$.
Example~\ref{ex_sum>1/2_x} illustrates the case $\max\{k,l,m\}\le 11$.

\smallskip
{\bf Case B}. {\it The condition $\max\{k,l,m\}\le 11$ fails}.
Recall our assumption $k\ge \max\{l,m\}$ due to symmetry.
Suppose that  $k\ge 12$. For such values of~$k$ the discriminant  $D(k)$ is positive.
However it does not mean that  $T(\theta)\ge 0$.
It may be $T(\theta)< 0$.
Firstly we show that all roots of $T$ belong to $(0, 1/k]$
for such values of  $k$.
Since there is no double roots, we can apply Sturm's method.
For this aim we construct on $[0,1/k]$  the Sturm's polynomial system of $T(\theta)$:
\begin{eqnarray*}
f_0(\theta)&:=&T(\theta), \\
f_1(\theta)&:=&\frac{dT}{d\theta}=2k(k-2)^2\theta-(k^2-4), \\
f_2(\theta)&:=&-\operatorname{rem}(f_0,f_1)=\frac{D(k)}{4\operatorname{lc}(T)}=\frac{k^2-12k+4}{4k(k-2)},
\end{eqnarray*}
where the symbol $\operatorname{rem}(f_0,f_1)$  means the remainder of dividing  $f_0$ by $f_1$ and  $\operatorname{lc}$ means the leading coefficient of a polynomial.
At the point $\theta=0$
the number of sign changes in the Sturm system equals~$2$ since  $k\ge 12$:
$$
f_0(0)=4>0, \qquad f_1(0)=4-k^2<0, \qquad f_2(0)>0.
$$

At $\theta=1/k$   the corresponding number is~$0$:
$$
f_0(1/k)=8/k>0, \qquad  f_1(1/k)=(k-2)(k-6)>0, \qquad  f_2(1/k)>0.
$$

Therefore  both  roots  of the polynomial
$T(\theta)$ belong to  the interval $(0,1/k)$ in cases $k\ge 12$.
Denote them $\Theta^{(1)}=\Theta^{(1)}(k)$ and
$\Theta^{(2)}=\Theta^{(2)}(k)$.
Actually we are able to find them explicitly:
$\Theta^{(1)}(k)=\dfrac{k+2-\sqrt{k^2-12k+4}}{2k(k-2)}$
and
$\Theta^{(2)}(k)=\dfrac{k+2+\sqrt{k^2-12k+4}}{2k(k-2)}$.
Note that
$$
0<\Theta^{(1)}(k)<\Theta^{(2)}(k)<\frac{1}{k}.
$$
Now the sign  of $\theta-\theta_1$ will depend on the fact
what link in the chain of these inequalities
 will contain~$\theta$.

\smallskip
{\it Subcase B1}.   $\theta\le \Theta^{(1)}$ or $\Theta^{(2)}\le \theta<1/k$.
Then  $\theta\ge \theta_1$.
According to Theorem~\ref{thm_sum_a_i>1/2}
no trajectory of~\eqref{three_equat} can leave~$R$ via the curve~$r_1$.
Examples~\ref{ex_sum>1/2_7} and~\ref{ex_sum>1/2_8} are relevant to {\it Subcase~B1}.

\smallskip
{\it Subcase B2}.   $\Theta^{(1)}<\theta<\Theta^{(2)}$ is equivalent to $\theta<\theta_1$.
By Theorem~\ref{thm_sum_a_i>1/2}  there exists a trajectory
of~\eqref{three_equat} that  can leave~$R$ but return back to~$R$.
Example~\ref{ex_sum>1/2_6} is relevant to {\it Subcase~B2}.

For further analysis the inequalities should be rewritten  in more convenient form in
{\it Subcases~B1 and~B2}\/
in terms of the parameters~$k,l$ and $m$.
For instance, {\it Subcase~B1} can be described by the union of solutions
of the inequalities $2<l+m\le X(k)$ and $l+m\ge Y(k)$,
where
 $$
X(k):=\frac{2k(k-2)}{k+2+\sqrt{k^2-12k+4}}-k+2, \qquad Y(k):=\frac{2k(k-2)}{k+2-\sqrt{k^2-12k+4}}-k+2
$$
are  functions  appeared in the text of Theorem~\ref{thm_3} and Table~2.

Imagine  $X(k)$ and $Y(k)$ as  functions of real independent variable $k$ to apply methods of calculus (see Fig.~\ref{pic2} for their graphs).
It is easy to establish $X(k)>2$ for all $k\ge 12$
(observe that $2$ is a horizontal asymptote since $\lim_{k \to \infty} X(k)/k=0$).
Moreover, $X(k)$ decreases  and satisfies $2<X(k)\le 5$ for $k\ge 12$
since $X(12)=5$ and
$$
X'(k)=-16\frac{k^2-5k+2-(k-1)\sqrt{k^2-12k+4}}{(k+2+\sqrt{k^2-12k+4})^2\sqrt{k^2-12k+4}}<0
$$
due to $(k^2-5k+2)^2-(k-1)^2(k^2-12k+4)=4k^3>0$.

As for $Y(k)$ it increases since
$$
Y'(k)=16\frac{k^2-5k+2+(k-1)\sqrt{k^2-12k+4}}{(k+2-\sqrt{k^2-12k+4})^2\sqrt{k^2-12k+4}}>0
$$
for all $k\ge 12$.
In addition,  $Y(k)\ge 10$ for all $k\ge 12$ with $Y(12)=10$.

\smallskip

{\it We claim that  only the values $k\in \{12, \dots,  16\}$ can
satisfy}\/ $\theta\ge \max\{\theta_1, \theta_2, \theta_3\}$  in~{\it Subcase~B1}.
Indeed since~$X(k)$ decreases and  $X(17)\approx 2.93$ with  the critical case $l+m\le 2$
it looses any sense to consider $2<l+m\le X(k)$ for $k\ge 17$: we would obtain $l=m=1$ for all $k\ge 17$ contradicting $l+m>2$.
Analogously due to increasing of  $Y(k)$  the numbers $l$ and $m$ in the sum $l+m\ge Y(k)$ will be out
of the range $\max\{l,m\}\le 11$ in a short time.
Then the inequalities
$$
\widetilde{Y}(k)\le l+m+k-2
$$
are required to be true for every permutation of the triple~$(k,l,m)$
in order
to keep the condition $\theta\ge \max\{\theta_1, \theta_2, \theta_3\}$
for~$k,l,m\ge 12$,
where $\widetilde{Y}(k):=\dfrac{2k(k-2)}{k+2-\sqrt{k^2-12k+4}}$.
It follows than  permissible values of  $l,m,k$  can not be large than~$16$.
Indeed   $k\ge \max\{l,m\}$ by assumption.
Since the function $\widetilde{Y}(k)=Y(k)+(k-2)$ increases for all $k\ge 12$
we obtain the following inequality
$$
\max\left\{\widetilde{Y}(k),\widetilde{Y}(l),\widetilde{Y}(m)\right\}= \widetilde{Y}(k)\le 3k-2.
$$
Introduce the function
$$
Z(k):=\widetilde{Y}(k)- (3k-2)=\frac{k^2+8k-4-(3k-2)\sqrt{k^2-12k+4}}{k+2-\sqrt{k^2-12k+4}}.
$$
Its  derivative $Z'(k)$ has the numerator
$k^3-6k^2-40k+16-(k^2-8k+8)\sqrt{k^2-12k+4}>0$
due to
$(k^3-6k^2-40k+16)^2-(k^2-8k+8)^2(k^2-12k+4)=16k^2(k^3-20k^2+104k-32)>0$
for all $k\ge 12$ (the unique real root of $k^3-20k^2+104k-32$ is less than $1$).
Therefore $Z(k)$ increases for $k\ge 12$.
Note also that $Z(16)\approx -0.0691<0$ and $Z(17)\approx 4.3137>0$
implying that the inequality  $\widetilde{Y}(k)\le 3k-2$ can not be satisfied if $k\ge 17$.

Collecting all values of $(k,l,m)$  which provide the condition $\theta\ge \max\{\theta_1, \theta_2, \theta_3\}$ of Theorem~\ref{thm_sum_a_i>1/2} in  both {\bf Cases A} and {\bf B}
we conclude that on GWS~{\bf 1} all metrics with $\operatorname{Ric}>0$ maintain $\operatorname{Ric}>0$
under NRF~\eqref{ricciflow}
if $(k,l,m)$ satisfies either $\max\{k,l,m\}\le 11$  or one of  the inequalities
$2<l+m\le X(k)$ or $l+m\ge Y(k)$ at $12\le k\le 16$.

\smallskip

$(2)$ The  system of inequalities
$X(k)<l+m<Y(k)$ and $k\ge 12$ corresponding to~{\it Subcase~B2}\/
is equivalent to the negation of~$\theta\ge \max\{\theta_1, \theta_2, \theta_3\}$
and guaranties the existence of some metrics preserving~$\operatorname{Ric}>0$
by the same Theorem~\ref{thm_sum_a_i>1/2}.

\smallskip

$(3)$
Let us estimate  the number of GWS~{\bf 1} on which
all metrics with $\operatorname{Ric}>0$ maintain $\operatorname{Ric}>0$.

It is obvious that there exists a finite number of GWS~{\bf 1}  which satisfy
$\max\{k,l,m\}\le 11$.
As we saw above  the inequality  $2<l+m\le X(k)$  corresponding to~{\it Subcase B1}\/ can  provide only a finite set of solutions in~$(k,l,m)$ due to decreasing of~$X(k)$.
For instance, $X(12)=5$ implies $l+m\le 5$ etc.
Analogously, another inequality  $Y(k)\le l+m$ corresponding to {\it Subcase B1}\/   admits a finite number of solutions $(k,l,m)$ as well
according to  assumed $\max\{l,m\}\le k$ and established  $12\le k\le 16$.
Therefore the number of GWS~{\bf 1}
on which any metric with $\operatorname{Ric}>0$ remains~$\operatorname{Ric}>0$
must be finite.

Obviously, there is no restriction to $k$
in the inequalities~$X(k)<l+m<Y(k)$ besides supposed $k\ge 12$ in {\bf Case B}.
Indeed they admit infinitely many solutions in $(k,l,m)$
since the range $(X(k),Y(k))$  expands as $k$ grows
providing an increasing number of pairs $(l,m)$ which can satisfy $X(k)<l+m<Y(k)$
for every fixed $k\ge 12$.
Therefore there are  infinitely (countably) many GWS~{\bf 1} on which at least some metrics can keep~$\operatorname{Ric}>0$.
Theorem~\ref{thm_3} is proved.
\end{proof}

\renewcommand{\arraystretch}{1.6}
\begin{table}[h]\label{tabXY}
{\bf Table 3}.  $X(k)$ and $Y(k)$ at $k\in \{12,\dots,17\}$
\begin{center}
\begin{tabular}{c|r|r|r|r|r|rr}
\hline
\hline
$k$&$12$& $13$ & $14$ & $15$ & $16$ & $17$ \\\hline
$X(k)$&$5$& $3.96$  &$3.51$  & $3.25$&  $3.07$&$2.94$\\
$Y(k)$&$10$& $15.29$     & $20.49$  & $26$&  $31.93$&$38.31$\\
\hline\hline
\end{tabular}
\end{center}
\end{table}

\begin{remark}\label{rem_140824}
All $\operatorname{GWS}$~{\bf 1}
on which any metric with $\operatorname{Ric}>0$ remains~$\operatorname{Ric}>0$
can easily be found.
Using values of $X(k)$ and~$Y(k)$ given in Table~3
all triples $(k, l, m)$ are found in Tables~4 and~5 which
satisfy  $2<l+m\le X(k)$ or $l+m\ge Y(k)$
at fixed $k\in \{12,\dots 16\}$ and $k\ge \max\{l,m\}$,
where $N=[Y(k)]$ if $Y(k)$ is integer, else $N=[Y(k)]+1$
(symbol $[x]$ means the integral part of $x$),
analogously  $\nu(l):=10-l$ if  $10-l>0$ else $\nu(l):=1$.
See also Fig.~\ref{XY12} and~\ref{XY14}.
The single pair $(l,m)=(16,16)$ corresponds to $k=16$.
No permissible pair $(l,m)$ at $k\ge 17$.
\end{remark}

\renewcommand{\arraystretch}{1.6}
\begin{table}[h]\label{tabX}
{\bf Table 4}. All triples $(k,l,m)$ which satisfy $2<l+m\le X(k)$
\begin{center}
\begin{tabular}{c|c|c|c|c|c|cc}
\hline
\hline
$k$&$12$& $13$ & $14$ & $15$ & $16$ & $\ge 17$ \\\hline
$(l,m)$&$(4,1), (1,4), (3,2), (2,3)$& $(2,1)$     & $(2,1)$  & $(2,1)$&  $(2,1)$&$-$   \\
&$(3,1), (1,3), (2,2)$& $(1,2)$  &$(1,2)$& $(1,2)$&$(1,2)$ &\\
&$(2,1), (1,2)$&           &&& &\\
\hline\hline
\end{tabular}
\end{center}
\end{table}

\renewcommand{\arraystretch}{1.6}
\begin{table}[h]\label{tabY}
{\bf Table 5}. All triples $(k,l,m)$ which satisfy $l+m\ge Y(k)$
\begin{center}
\begin{tabular}{c|c|c|c|cc}
\hline
\hline
$k$&$12$& 13, 14, 15 & $16$ & $\ge 17$ \\\hline
$(l,m)$& $1 \le l \le 12$  & $N-k\le l \le k$ &  $(16,16)$&$-$   \\
&$\nu(l) \le m\le  12$ &$N-l\le m \le  k$&  &   \\
\hline\hline
\end{tabular}
\end{center}
\end{table}

\bigskip

\begin{proof}[Proof of Corollary~\ref{corol_1}]
Recall that we are using the list of generalized Wallach spaces
given in Table~1 in accordance with~\cite{Nikonorov4}.

\smallskip

$(1)$  GWS {\bf 2, 5, 7}.
The proof follows from Theorem~\ref{thm_sum_a_i<1/2}
since  $a_1+a_2+a_3=1/2$ for spaces {\bf 2, 5, 7}.
An illustration is given in Example~\ref{ex_sum_a_i<1/2}.

\smallskip

$(2)$ GWS {\bf 3, 15}.
Clearly $\theta=a_1+a_2+a_3-\dfrac{1}{2}=-\dfrac{1}{2(k+l+m+1)}<0$  for
 GWS~{\bf 3}.
Analogously $\theta=-1/6<0$ for GWS~{\bf 15}.
The proof follows from  Theorem~\ref{thm_sum_a_i<1/2}.

\smallskip

$(3)$ GWS {\bf 4, 6, 8, 9, 10, 11, 12, 13, 14}.
For all them $\theta>0$. Desired conclusions follow by  checking the conditions
$\max\left\{\theta_1, \theta_2, \theta_3\right\}<\theta<1$ directly.
We demonstrate the proof only for GWS~{\bf 4} which depend
on parameters.
It is  easy to show that  $\theta=a_1+a_2+a_3-\dfrac{1}{2}=\dfrac{1}{4}$  is greater than
each of
${\theta}_1=\frac{2l\sqrt{(3l+1)(l-1)}-3l^2+2l+1}{4l(3l+1)}$,
${\theta}_2=\frac{2l\sqrt{(l+1)(3l-1)}-3l^2-2l+1}{4l(3l-1)}$ and
${\theta}_3=\frac{\sqrt{3}}{6}-\frac{1}{4}$
for all $l=2,3,\dots$.
Corollary~\ref{corol_1} is proved.
\end{proof}

\section{Planar illustrations of results}

\begin{figure}[h]
\centering
\includegraphics[width=0.9\textwidth]{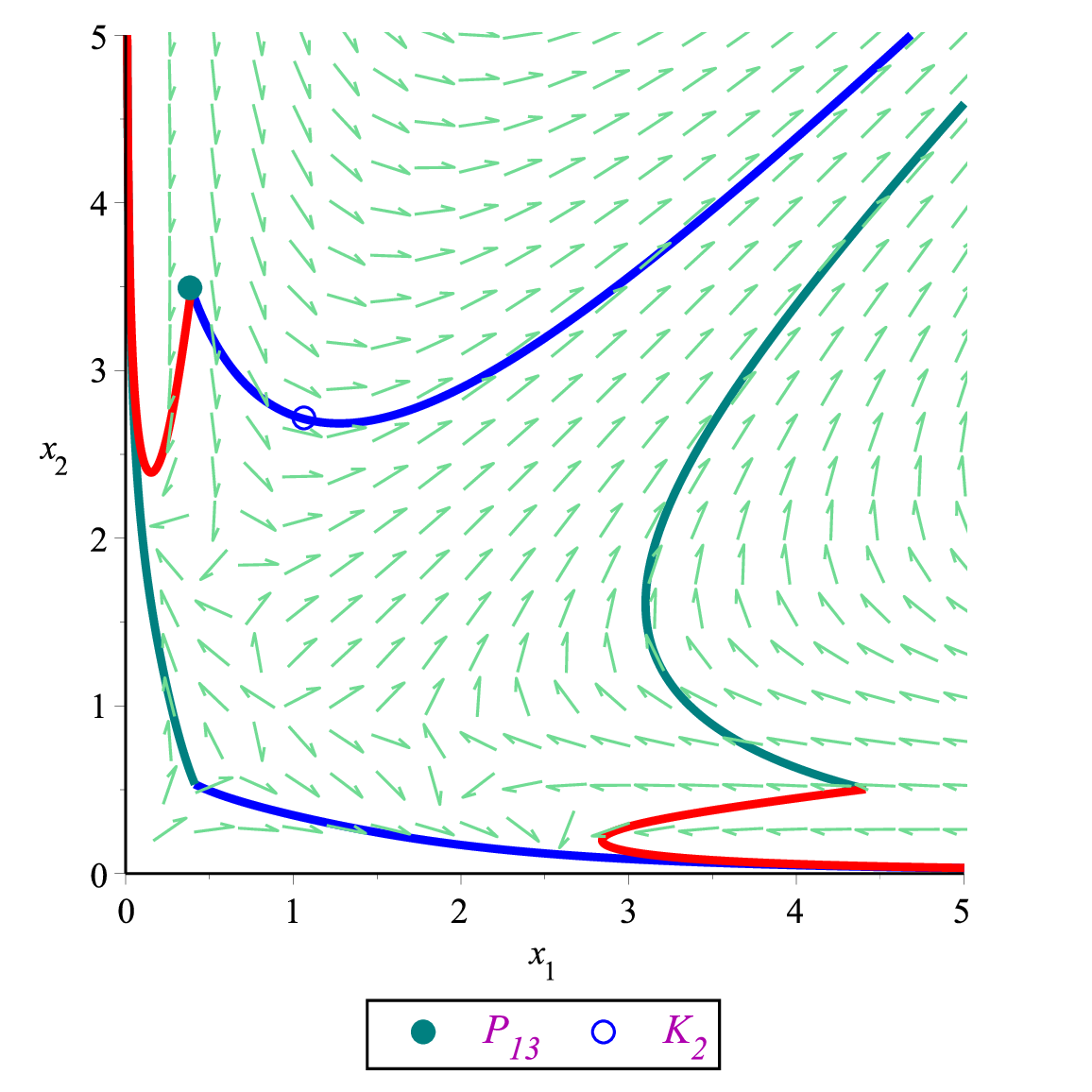}
\caption{GWS~{\bf  3} with $k=5$, $l=4$, $m=3$:
the phase portrait of~\eqref{two_equat} and the projections of~$r_1,r_2,r_3$ onto the plane $x_3=0$}
\label{pic3}
\end{figure}

\begin{figure}[h]
\centering
\includegraphics[width=0.9\textwidth]{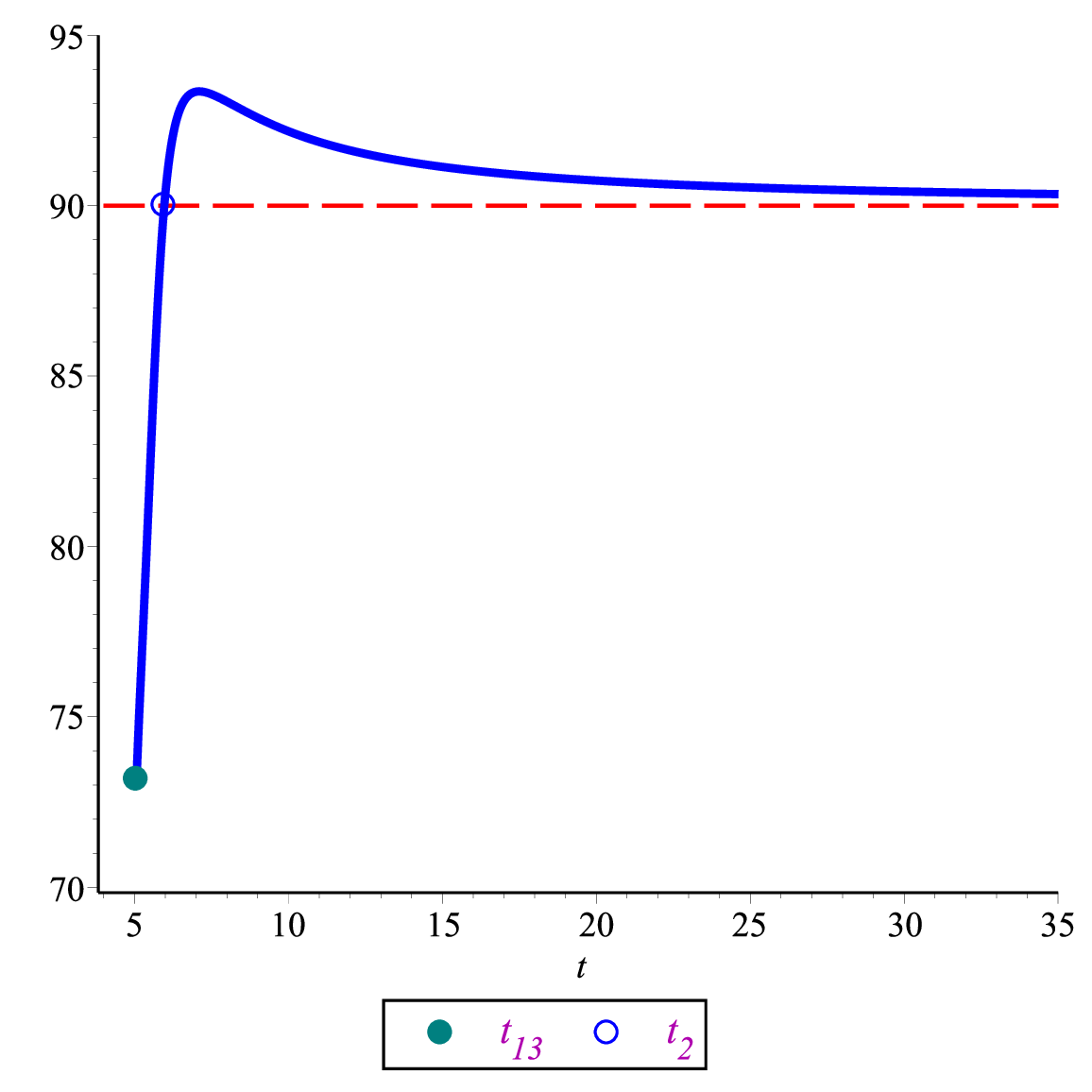}
\caption{GWS~{\bf  3} with $k=5$, $l=4$, $m=3$: the function~$\alpha(t)$ on the component~$r_{13}$}
\label{pic4}
\end{figure}

\subsection{The case $a_1+a_2+a_3\le 1/2$}

\begin{example}\label{ex_sum_a_i<1/2}
$\operatorname{Ric}>0$ is not preserved on every $\operatorname{GWS}$~{\bf  3}.
Take for instance $(k,l,m)=(5, 4, 3)$.
Then $a_1=5/26\approx 0.1923$, $a_2=2/13\approx 0.1538$ and $a_3=3/26\approx 0.1154$ with $a_1+a_2+a_3<1/2$.
Results are depicted in Fig.~\ref{pic3} and Fig.~\ref{pic4}.
\end{example}

In Example~\ref{ex_sum_a_i<1/2} we demonstrate all evaluations in detail.
In sequel we will restrict ourselves to concise comments.
The curve~$r_1=\Sigma\cap \Lambda_1$ defined by the system
of equations $5\left(x_1^2-x_2^2-x_3^2\right)+26x_2x_3=0$ and
$x_1^{1/5}\,x_2^{1/4}\,x_3^{1/3}=1$
has the following parametric representation according to Lemma~\ref{param_r_i}:
$x_1=\phi_1(t)=\frac{\psi^{35/94}}{t^{15/47}}$,
$x_2=\phi_2(t)=\frac{t^{32/47}}{\psi^{6/47}}$ and
$x_3=\phi_3(t)=\frac{t^{-15/47}}{\psi^{6/47}}$,
where $\psi:=t^2-26t/5+1$, $t\in (0,m_1)\cup (M_1,+\infty)$, $m_1=1/5$, $M_1=5$.

Choose the component~$r_{13}$ corresponding to  $t\in (5,+\infty)$.
Coordinates of the point~$P_{13}=r_{13}\cap r_{31}$  can be found using formulas~\eqref{alpha_beta} and~\eqref{Pij}:
$x_1=\gamma q^0\approx 0.3916$, $x_2=q^0\approx 3.4857$, $x_3=\delta q^0\approx 17.4285$,
where
$\gamma=\Phi(a_1,a_3)=\frac{-27+9\sqrt{10}}{13}\approx 0.1123$,\;
$\delta=\Psi(a_1,a_3)=\frac{50-15\sqrt{10}}{13} \approx 0.1974$, \;
$q^0=\gamma^{-\omega a_1^{-1}}\delta^{-\omega a_3^{-1}}=
\gamma^{-12/47}\delta^{-15/47}\approx 3.4857$.
Due to formula~\eqref{t_ij_a}   the following value
of the parameter~$t$ corresponds to the point~$P_{13}$:
 $t=t_{13}=\delta^{-1}=\frac{13}{50-15\sqrt{10}} \approx 5.0666$.

The polynomial $h(t)$ in~\eqref{Hgoi} takes the form
$h(t)= -63375t^4-370500t^3+4529574t^2-370500t-63375$.
Respectively
$p(y) = -63375y^2-370500y+4656324$
has a unique positive root
$y_0=-\frac{38}{13}+\frac{192}{325}\sqrt{235}\approx 6.1332>a_1^{-1}=5.2$.
The corresponding roots of $h(t)$ are the following:
$t_1\approx 0.1676<m_1=0.2$ and $t_2\approx 5.9656>M_1=5$.

The vector field ${\bold V}$ changes its direction at the unique
point~$K_2$ with coordinates
$x_1=\phi_1(t_2)\approx 1.0717$,
$x_2=\phi_2(t_2)\approx 2.7097$,
$x_3=\phi_3(t_2)\approx 0.4542$.

For $P_{13}$ we have $M_1 <t_{13}<t_2$.
Therefore $h(t)>0$  on~$r_{13}$ if  $t\in [t_{13},t_2)$.
This means that trajectories originated in the exterior of~$R$ can reach~$R$
across the arc of the curve~$r_{13}$ with endpoints at $P_{13}$ and $K_2$.
In sequel the mentioned incoming trajectories
and trajectories  originated in~$R$,
will leave~$R$ at once through an unbounded part of the curve~$r_{13}$,
which extends from $K_2$ to infinity
as the image of the interval $(t_2, +\infty)$.
Such a conclusion is confirmed by the values of the limits in~\eqref{parcurlims} as well:
$\lim\limits_{t\to M_1+} x_1=0$,
$\lim\limits_{t\to M_1+} x_2= \lim\limits_{t\to M_1+} x_3 = +\infty$,
$\lim\limits_{t\to +\infty} x_1=\lim\limits_{t\to +\infty} x_2 = +\infty$,
$\lim\limits_{t\to +\infty} x_3 =0$.

Using coordinates of the normal found in Lemma~\ref{normal_inside_gen},
we can analyze values of the angle
$$
\alpha(t):=\frac{180}{\pi}\arccos\, \frac{({\bold V}, \nabla \lambda_1)}{\|{\bold V}\|\,\|\nabla \lambda_1\|}
$$
between ${\bold V}=(f,g,h)$ and $\lambda_1$ on every point of the curve $r_{13}$, in degrees.
As calculations show  $73.16^{\circ}<\alpha(t)< 90^{\circ}$  for the incoming flow at $t\in (t_{13}, t_2)$
and $90^{\circ}<\alpha(t)<93.36^{\circ}$ for the outgoing flow at  $t>t_2$.
Therefore, trajectories  leave $R$ under small angles not exceeding~$4^{\circ}$.

\begin{example}\label{ex_sum_a_i=1/2}
$\operatorname{Ric}>0$ is not preserved on every  $\operatorname{GWS}$~{\bf  2}.
Take for instance $k=5$, $l=4$ and $m=3$.
Then $a_1=5/24\approx 0.2083$, $a_2=4/24\approx 0.1667$ and $a_3=3/24=0.125$ with $a_1+a_2+a_3=1/2$.
\end{example}

All necessary evaluations can  be repeated in the same order as in Example~\ref{ex_sum_a_i<1/2}.
Observe that $71.62<\alpha(t)\le 90$ for incoming trajectories
at
$4.6533\approx\frac{64}{5(45-\sqrt{595}\sqrt{3})}=t_{13}\le t\le t_2=\frac{3713+17\sqrt{42721}}{1200}\approx 6.0223$
and $90<\alpha(t)<91.43$ for outgoing trajectories  at  $t>t_2$.
Trajectories leave~$R$ under angles not greater than~$1.5^{\circ}$,
that means they are almost parallel to tangent vectors to~$r_{13}$.

\subsection{The case $a_1+a_2+a_3>1/2$}

\begin{figure}[h!]
\centering
\includegraphics[width=0.9\textwidth]{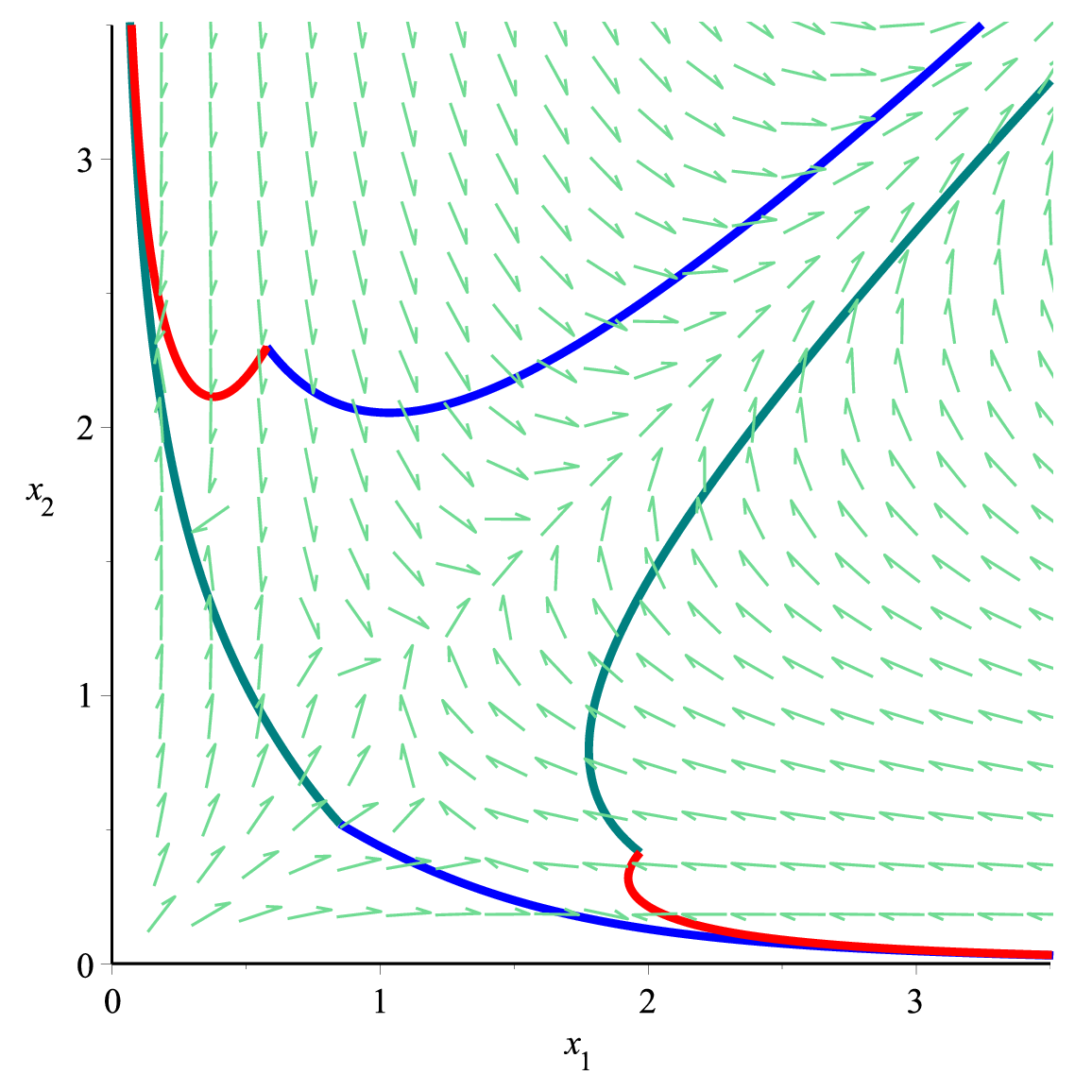}
\caption{The phase portrait of~\eqref{two_equat} at $a_1=5/14$, $a_2=3/7$, $a_3=\frac{1}{14}+\frac{\sqrt{6}}{12}$}
\label{pic9}
\end{figure}

\begin{example}\label{ex_sum>1/2_4}
The values $a_1=5/14 \approx 0.3571$,  $a_2=3/7\approx 0.4286$ and $a_3=\frac{1}{14}+\frac{\sqrt{6}}{12}\approx 0.2756$
satisfy $a_1+a_2+a_3>1/2$ and $\theta\ge \max\left\{\theta_1, \theta_2, \theta_3\right\}$.
Therefore trajectories of~\eqref{three_equat} are expected only incoming into $R$ by Theorem~\ref{thm_sum_a_i>1/2} (see also Fig.~\ref{pic9}).
\end{example}

Indeed $\theta=\frac{5}{14}+\frac{\sqrt{6}}{12}\approx 0.5613$,
$\theta_1=-\frac{1}{7}+\frac{\sqrt{6}}{12}\approx 0.0613$,
$\theta_2=-\frac{1}{14}+\frac{\sqrt{13}}{26}\approx 0.0672$,
$\theta_3=\frac{1}{14}\frac{-239+14\sqrt{6}+7\sqrt{1434-84\sqrt{6}}}{48+7\sqrt{6}}\approx 0.0445$.
It should be noted that in general such artificial values of the parameters $a_1,a_2,a_3$
need not correspond to certain generalized Wallach space  as noted in Introduction.

\begin{figure}[h!]
\centering
\includegraphics[width=0.9\textwidth]{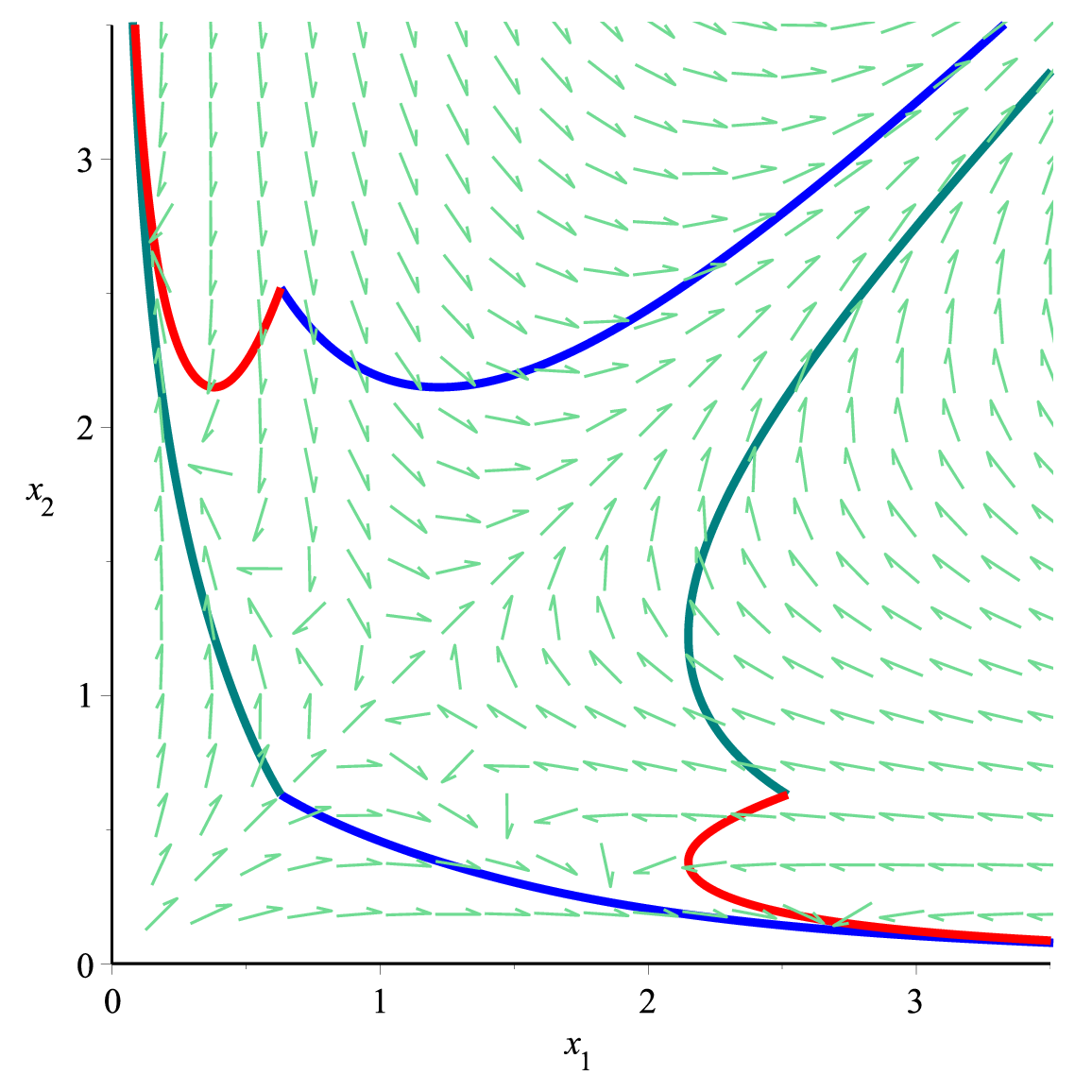}
\caption{GWS~{\bf  1} with $k=l=m=2$: the phase portrait of~\eqref{two_equat}}
\label{pic5}
\end{figure}

\begin{figure}[h!]
\centering
\includegraphics[width=0.9\textwidth]{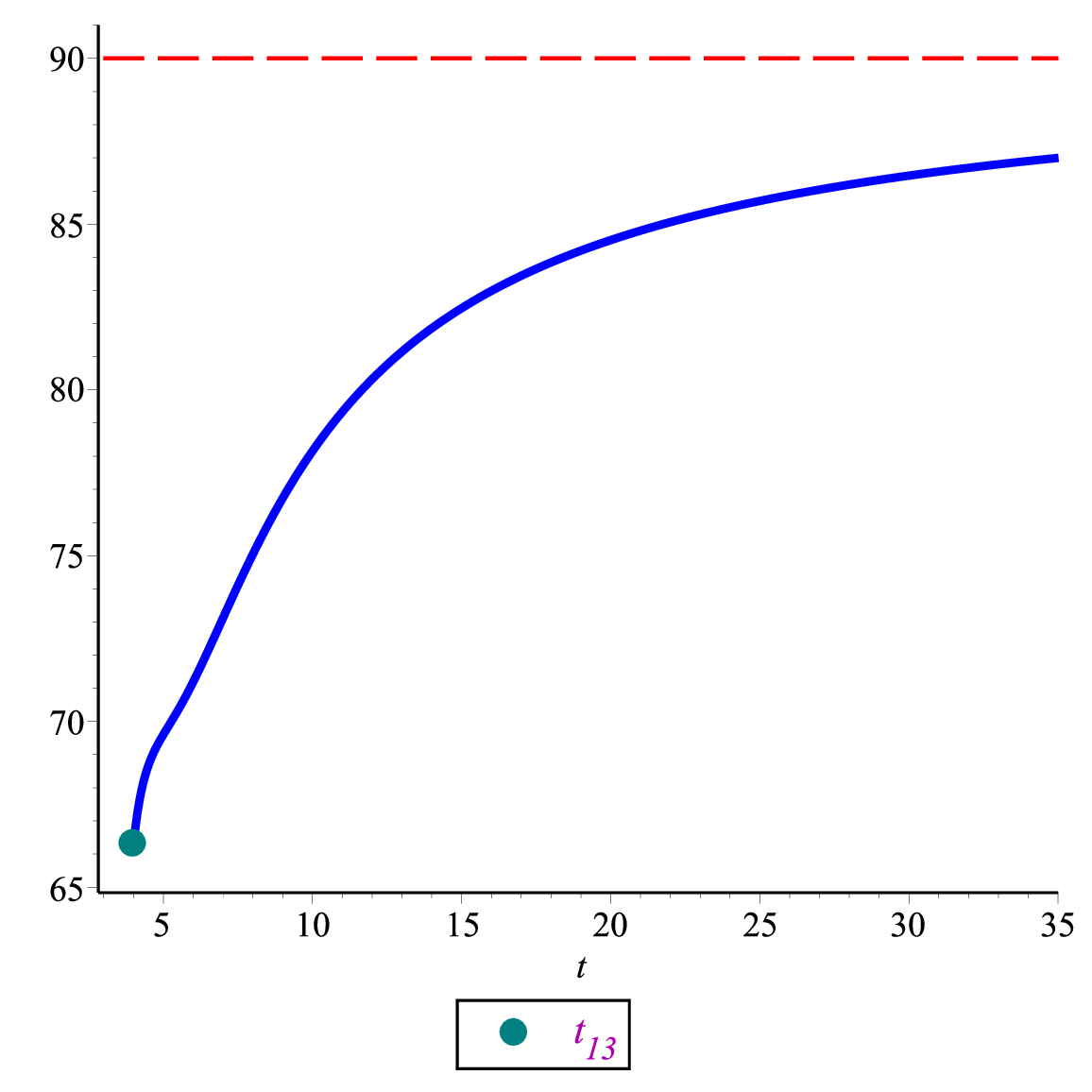}
\caption{GWS~{\bf  1} with $k=l=m=2$:
 the function~$\alpha(t)$ on the component $r_{13}$}
\label{pic6}
\end{figure}

\begin{example}\label{ex_sum>1/2_x}
On $\operatorname{GWS}$~{\bf  1} with  $k=l=m=2$ ($a_1=a_2=a_3=a=1/4$)
all metrics with $\operatorname{Ric}>0$ preserve $\operatorname{Ric}>0$.
Results are depicted in Fig.~\ref{pic5} and Fig.~\ref{pic6}.
\end{example}

Since $\max\{k,l,m\}\le 11$ the sufficient condition
$\theta\ge \max\left\{\theta_1, \theta_2, \theta_3\right\}$
is satisfied.
Indeed $\theta=0.25$ and
$\theta_1=\theta_2=\theta_3\approx 0.0387$.
Therefore we have only trajectories incoming into~$R$.
Note that  this example confirms also Theorem~\ref{Sect_Ricci_gen} (Theorem~3 in \cite{AN}),
because  $a=1/4>1/6$.

\begin{figure}[h!]
\centering
\includegraphics[width=0.9\textwidth]{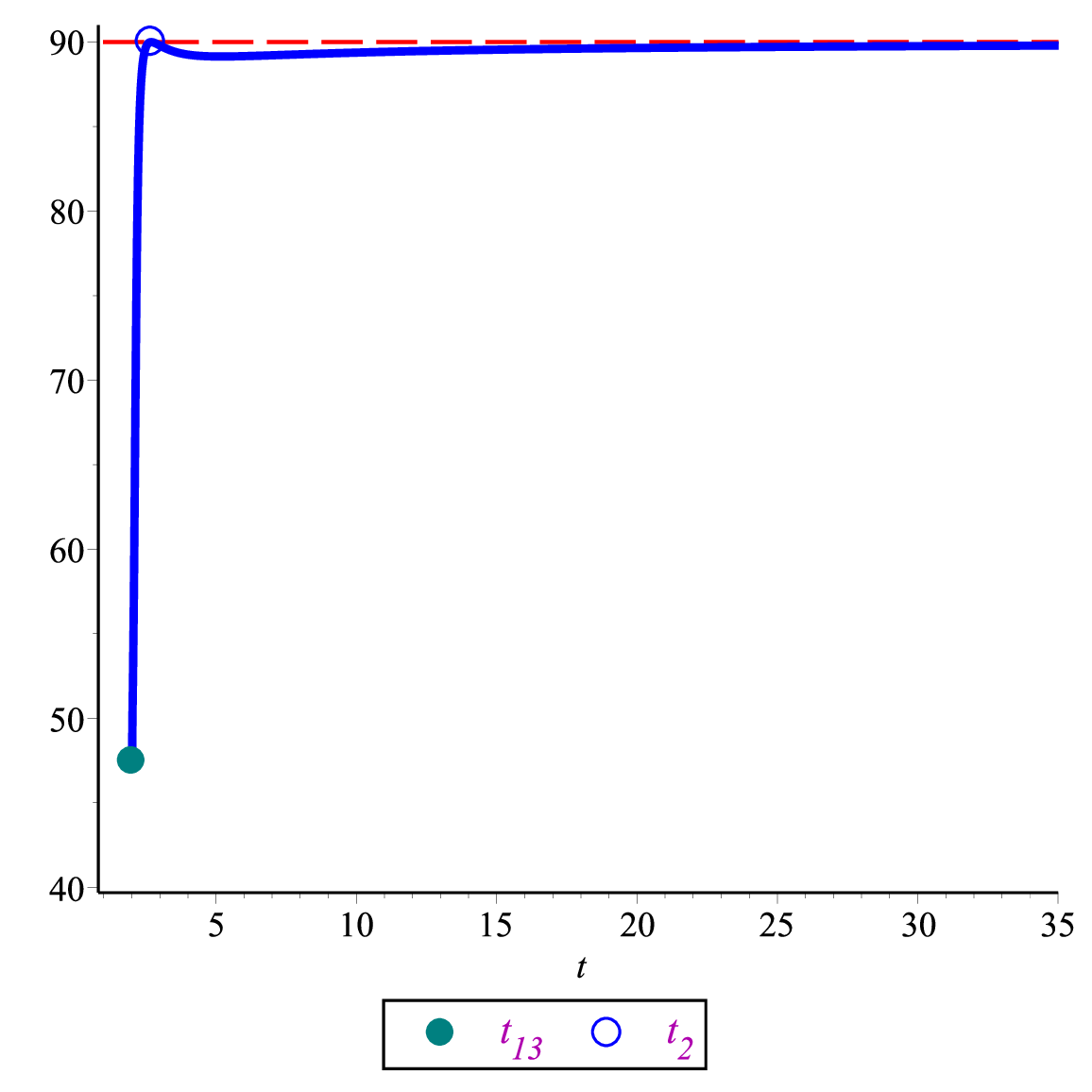}
\caption{GWS~{\bf  1} with $k=12$, $l=3$, $m=2$: the function $\alpha(t)$
on the component $r_{13}$}
\label{pic8}
\end{figure}

\begin{example}\label{ex_sum>1/2_7}
On $\operatorname{GWS}$~{\bf  1} with $k=12$, $l=3$, $m=2$ ($a_1=2/5$,
$a_2=1/10$, $a_3=1/15\approx 0.0667$)
 all metrics with $\operatorname{Ric}>0$  preserve $\operatorname{Ric}>0$.
\end{example}

The point $(3,2)$ lies on the straight line $l+m=X(12)=5$.
Therefore $\theta=\theta_1$  meaning that  trajectories  can only touch
each component of~$r_1$ at some unique point  of that component
remaining inside of~$R$ (see Fig.~\ref{pic8}).
Also trajectories of~\eqref{three_equat} can not leave $R$ via $r_2$ or~$r_3$ since
$\theta>\max\{\theta_2,\theta_3\}$ due to $\max\{l,m\}<11$.
Therefore they remain in~$R$.

In detail, the polynomial $p(y)=(16y-49)^2$ corresponds to~$r_1$
that admits the single root $y_0=49/16=3.0625$ of multiplicity~$2$.
Then the corresponding  polynomial
$h(t)=\left(16t^2-49t+16\right)^2$
admits two roots $t_1=\frac{49-9\sqrt{17}}{32}\approx 0.3716$ and
$t_2=\frac{49+9\sqrt{17}}{32}\approx 2.6909$ each of multiplicity $2$.

Actual values of the parameters are
$\theta=\theta_1=1/15\approx 0.0667$,
$\theta_2=-2/5+\sqrt{6}/6\approx 0.0082$,
$\theta_3=-13/30+3\sqrt{221}/34\approx 0.0039$.

\begin{example}\label{ex_sum>1/2_8}
On $\operatorname{GWS}$~{\bf  1} with $k=15$, $l=14$, $m=13$
($a_1=3/16=0.1875$, $a_2=7/40=0.175$, $a_3=13/80=0.1625$)
all metrics with $\operatorname{Ric}>0$ preserve $\operatorname{Ric}>0$.
\end{example}

The triple $(15, 14,13)$ satisfies $l+m>Y(15)=26$.
Therefore $\theta>\theta_1$.
Analogously the inequalities $k+m>Y(14)$ and $l+k>Y(13)$  imply
$\theta>\theta_2$ and $\theta>\theta_3$,
where  $Y(14)\approx 20.49$ and $Y(13)\approx 15.29$
are known from Table~\ref{tabXY}.
Therefore no trajectory of~\eqref{three_equat} can leave $R$.

Supplementary, for
each  $r_1, r_2$ and $r_3$ their polynomials of the kind~$h(t)$ admit no real root
preserving positive sign.
Actual values of the parameters are
$\theta=1/40= 0.025$,  $\theta_1=-5/16+\sqrt{55}/22\approx 0.0246$,
$\theta_2=-13/40+\sqrt{39}/18\approx 0.0219$,
$\theta_3=-27/80+3\sqrt{159}/106\approx 0.0194$.

\begin{figure}[h!]
\centering
\includegraphics[width=0.9\textwidth]{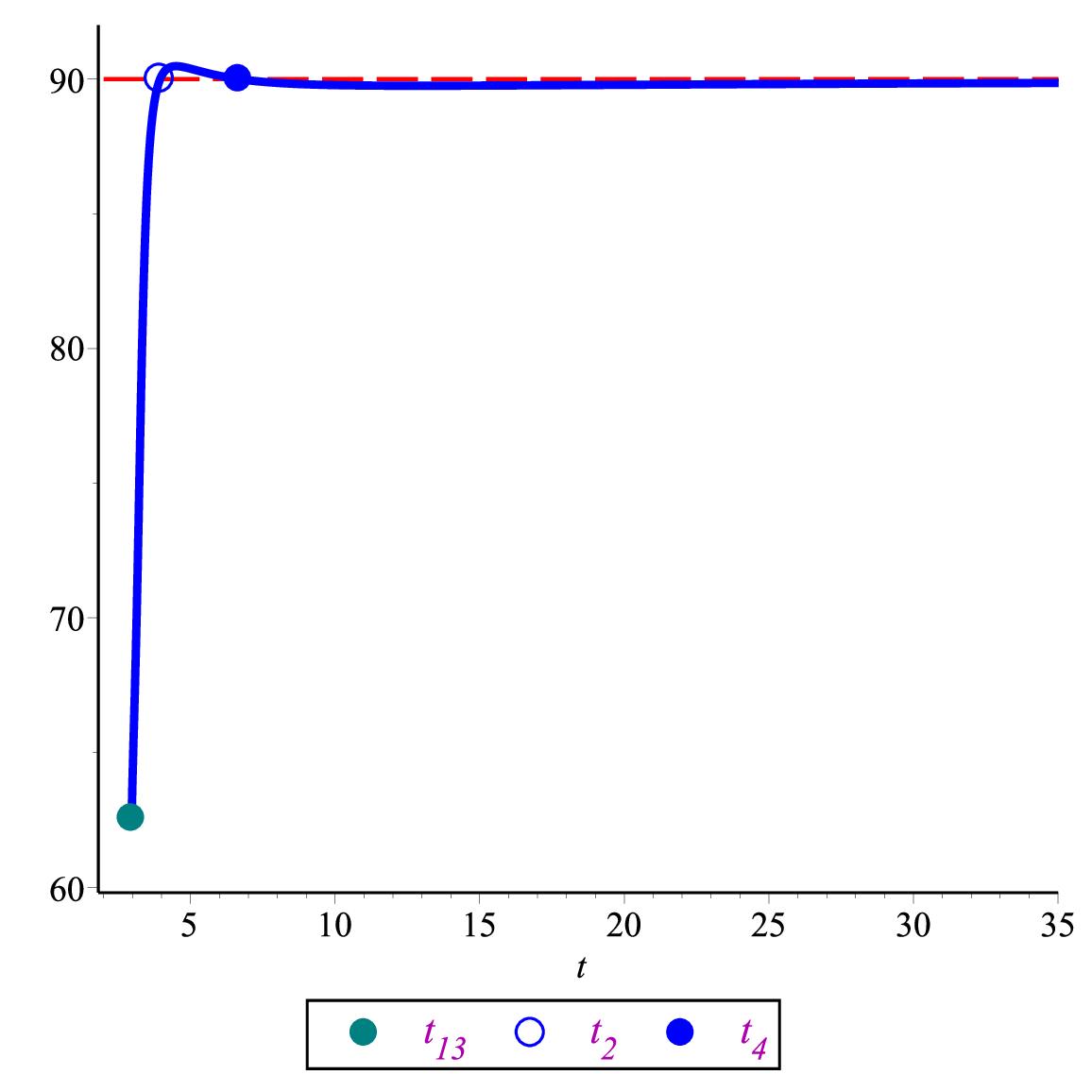}
\caption{GWS~{\bf  1}  with $k=14$, $l=7$, $m=4$:  the function $\alpha(t)$
on the component $r_{13}$}
\label{pic7}
\end{figure}

\begin{figure}[h!]
\centering
\includegraphics[width=0.9\textwidth]{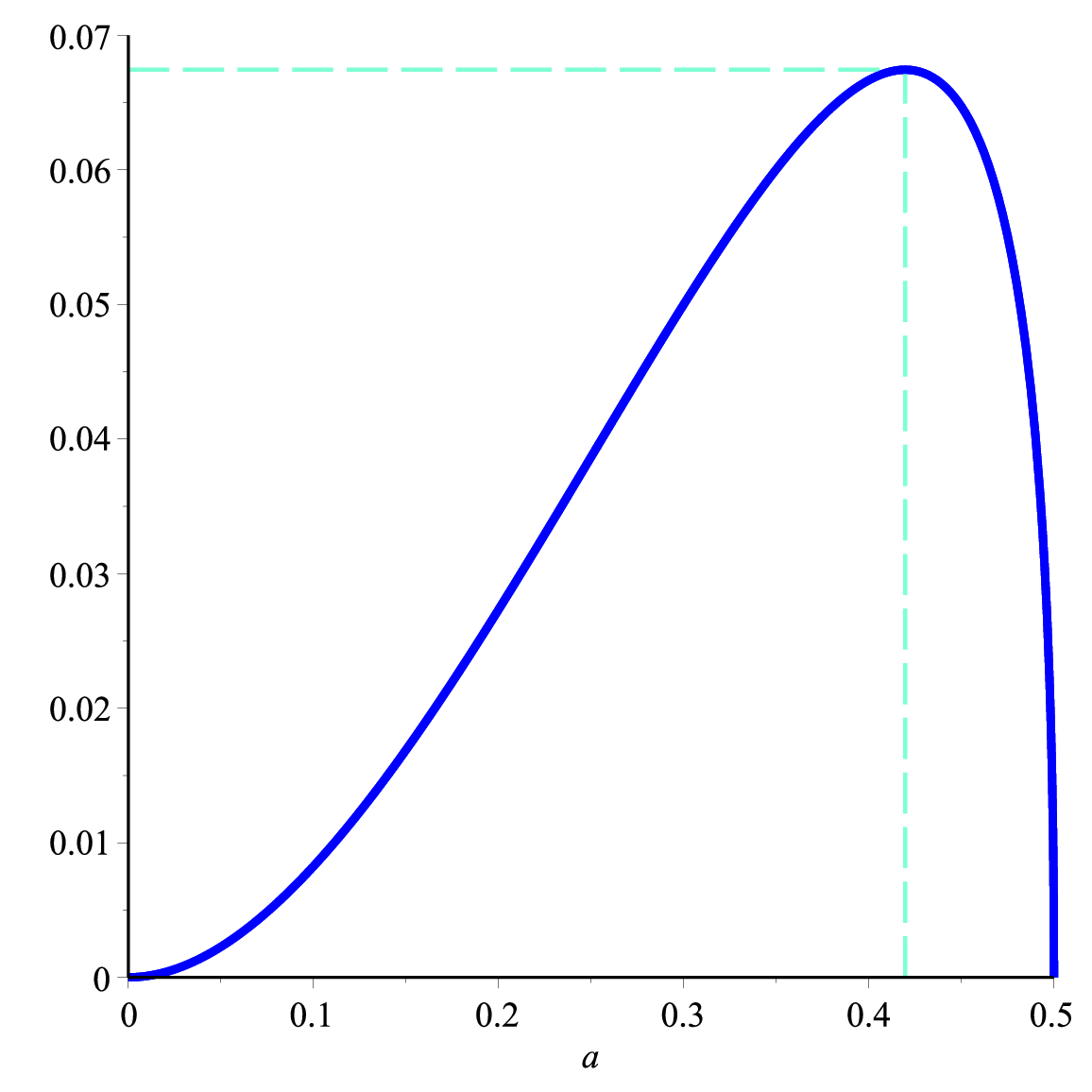}
\caption{The function $\theta_i=\theta_i(a)$, $a\in (0,1/2)$.}
\label{pic11}
\end{figure}

\begin{example}\label{ex_sum>1/2_6}
On $\operatorname{GWS}$~{\bf  1} with  $k=14$, $l=7$, $m=4$
($a_1=7/23\approx 0.3043$, $a_2=7/46\approx 0.1522$, $a_3=2/23\approx 0.0869$)
some metrics  with $\operatorname{Ric}>0$  preserve $\operatorname{Ric}>0$.
\end{example}

Since $l, m<11$ and  $(k,l,m)=(14,7,4)$ satisfies $X(k)\approx 3.51 <l+m<Y(k)\approx 20.49$
then $\theta>\theta_2$, $\theta>\theta_3$ and $\theta<\theta_1$
avoiding direct computations
(actually
$\theta=1/23\approx 0.0435$,  $\theta_1=-9/46+3\sqrt{37}/74\approx 0.0509$,
$\theta_2=-8/23+\sqrt{30}/15\approx 0.0173$,
$\theta_3=-19/46+\sqrt{57}/18\approx 0.0064$).
Therefore no trajectory can leave~$R$ via $r_2, r_3$
but  there expected exactly two points on each component of~$r_1$
at which trajectories of~\eqref{three_equat} can change directions relatively to $R$ in accordance
with the scenario \glqq towards~$R$ - from~$R$ - towards~$R$ again``
(see Fig.~\ref{pic7}).

Supplementary, the polynomial~$h(t)$ corresponding to $r_1$ admits four roots
$t_1\approx 0.2532$, $t_2\approx 3.9488$,
$t_3\approx 0.15$, $t_4\approx 6.6663$ such that
$t_3<t_1<t_{12}=\frac{320}{161}-\frac{8\sqrt{1110}}{161}\approx 0.3321$
and
$t_{13}=\frac{230}{7(\sqrt{2109}-57)}\approx 2.9665<t_2<t_4$.

\section{Final remarks}

${\bf 1)}$ Let us return to generalized Wallach spaces with $a_1=a_2=a_3=a$,
in particular, let us consider the cases
$a=1/4$, $a=1/6$, $a=1/8$ and $a=1/9$ studied in~\cite{Ab7, Ab_RM, AN}.
As mentioned in Introduction the Wallach spaces $W_{6}$, $W_{12}$ and $W_{24}$
are examples of GWS which correspond to $a=1/6$, $a=1/8$ and $a=1/9$ respectively.
The value $a=1/4$ is very special at which four different singular points of~\eqref{three_equat} merge to its unique linearly zero saddle, see~\cite{Ab7, AANS1}. Moreover,  on the
algebraic surface of degree~$12$ mentioned in Introduction,
the point $(1/4,1/4,1/4)$ is a singular point of elliptic umbilic type in the sense of Darboux (see~\cite{Ab2, AANS1}).

Theorem~\ref{thm_3} generalizes Theorem~\ref{thm1}.
Indeed generalized Wallach spaces with $a=1/4$
can be obtained in the special case of GWS~{\bf 1} with $k=l=m=2$.
The condition $\max\{k,l,m\}\le 11$ is satisfied.
Therefore by Theorem~\ref{thm_3} all  initial metrics
with $\operatorname{Ric}>0$
maintain $\operatorname{Ric}>0$.

Although Theorem~\ref{thm_sum_a_i<1/2} covers Theorem~\ref{Sect_Ricci_gen}
 we have stronger statements in~Theorem~\ref{Sect_Ricci_gen} that all  initial metrics with $\operatorname{Ric}>0$
lose $\operatorname{Ric}>0$ on the narrow class of GWS with $a\in (0,1/6)$,
in particular on the spaces~$W_{12}$ and~$ W_{24}$,
whereas Theorem~\ref{thm_sum_a_i<1/2} states such a property only for some initial metrics with $\operatorname{Ric}>0$, but for a wider class of GWS with $a_1+a_2+a_3\le 1/2$,
in particular, for the mentioned $a=1/8$ and $a= 1/9$ ($a_1+a_2+a_3<1/2$),
which correspond to GWS~{\bf 3} with $k=l=m=1$ and GWS~{\bf 15} respectively.

Analogously, for $a\in (1/6,1/4)\cup (1/4,1/2)$ all metrics become
$\operatorname{Ric}>0$ according to Theorem~\ref{Sect_Ricci_gen},
whereas by Theorem~\ref{thm_sum_a_i>1/2} (even by Theorem~\ref{thm_3})
the similar property can be guarantied
for some metrics only, but again on a wider class of GWS
which satisfy $a_1+a_2+a_3>1/2$.

Theorem~\ref{thm_sum_a_i<1/2} is consistent with Theorem~\ref{Sect_Ricci_genn} and extends it.
According to Theorem~\ref{Sect_Ricci_genn} not every metric lose~$\operatorname{Ric}>0$ in the case $a=1/6$ (they are exactly metrics  contained in the domain bounded by parameters of K\"ahler metrics $x_k=x_i+x_j$,  $\{i,j,k\}=\{1,2,3\}$, see~\cite{AN}).
This fact does not contradict the conclusions of Theorem~\ref{thm_sum_a_i<1/2},
where spaces with $a=1/6$ are contained as a special class of GWS~{\bf 2} with $k=l=m$ ($a_1+a_2+a_3=1/2$).
Moreover, Theorem~\ref{thm_sum_a_i<1/2} implies the departure of some metrics
from the domain with $\operatorname{Ric}>0$ at $a=1/6$, that was not previously stated in Theorem~\ref{Sect_Ricci_genn}.

\medskip

${\bf 2}$) In general there is infinitely (uncountably) many triples
$(a_1,a_2,a_3)\in (0,1/2)^3$ satisfying $a_1+a_2+a_3>1/2$
such that trajectories of~\eqref{three_equat} all remain in~$R$ originating there.
Indeed each  function
$\theta_i=\theta_i(a_i)=a_i-\frac{1}{2}+\frac{1}{2}\sqrt{\frac{1-2a_i}{1+2a_i}}$, $i=1,2,3$,
attains its largest value $\theta^{\ast}_i=\theta_i(a^{\ast})\approx 0.067442248$
at the point
$
a^{\ast}=\frac{\mu^2-2\mu+4}{6\mu}\approx 0.4196433778
$
where $\mu=\sqrt[3]{19+3\sqrt{33}}\approx 3.30905648$ (see Fig.~\ref{pic11}).
Then we always have an opportunity to get $\theta >\theta_i$ for each~$i$
choosing (uncountably many) $a_i$  at least from the interval $[a^{\ast}, 1/2)$
because
$$
\theta=a_1+a_2+a_3-0.5\ge 3a^{\ast}-0.5>1.2-0.5=0.7>\theta_i.
$$
Therefore the number of appropriated generalized Wallach spaces, where all metrics with $\operatorname{Ric}>0$  maintain  $\operatorname{Ric}>0$,
might be unbounded too if they exist for such $(a_1,a_2,a_3)$
(compare with the third assertion of Theorem~\ref{thm_3}).

\bigskip

The author expresses his gratitude to Professor Yuri\u\i\ Nikonorov for helpful discussions on the topic of this paper.

\vspace{10mm}

\bibliographystyle{amsunsrt}

\begin{thebibliography}{[99]}


\bibitem{Ab1}
 Abiev, N.\,A.: Two-parametric bifurcations of singular points of
the normalized Ricci flow on generalized Wallach spaces.
AIP Conf. Proc. {\bf 1676} (020053), 1--6 (2015)
https://doi.org/10.1063/1.4930479



\bibitem{Ab2}
 Abiev, N.\,A.: On topological structure of some sets related to
the  normalized Ricci flow on generalized Wallach spaces.
Vladikavkaz. Math. J. {\bf 17}(3), 5--13 (2015)
https://doi.org/10.23671/VNC.2017.3.7257


\bibitem{Ab7}
 Abiev, N.\,A.:
On the evolution of invariant Riemannian metrics on one class
of generalized Wallach spaces under the influence of the normalized Ricci flow.
Mat. Tr. {\bf 20}(1), 3--20 (2017) (Russian).
[English translation. In: Sib. Adv. Math. {\bf 27}(4), 227--238 (2017)]
https://doi.org/10.3103/S1055134417040010

\bibitem{Ab_RM}	
Abiev, N.: On the dynamics of a three-dimensional  differential system related to the normalized Ricci flow on generalized Wallach spaces.  Results  Math.  {\bf 79}, 198 (2024).  https://doi.org/10.1007/s00025-024-02229-w


\bibitem{AANS1}
 Abiev, N.\,A., Arvanitoyeorgos, A., Nikonorov, Yu.\,G., Siasos, P.:
The dynamics of the Ricci flow on generalized Wallach spaces.
Differ. Geom. Appl. {\bf 35}(Suppl.), 26--43 (2014)
https://doi.org/10.1016/j.difgeo.2014.02.002

\bibitem{AANS2}
Abiev, N.\,A., Arvanitoyeorgos, A., Nikonorov, Yu.\,G., Siasos, P.:
The Ricci flow on some generalized Wallach spaces. In: V. Rovenski, P. Walczak (eds.)
Geometry and its Applications. Springer Proceedings in Mathematics \&  Statistics. {\bf 72}, pp. 3--37. Springer, Switzerland (2014)
https://doi.org/10.1007/978-3-319-04675-4

\bibitem{AN}
 Abiev, N.\,A.,  Nikonorov, Yu.\,G.:
The evolution of positively curved invariant Riemannian metrics on the Wallach spaces under the Ricci flow. Ann.~Glob.~Anal.~Geom. {\bf 50}(1), 65--84. (2016)
https://doi.org/10.1007/s10455-016-9502-8

\bibitem{AW}
Aloff, S., Wallach, N.\,R.:
An infinite family of 7--manifolds admitting positively curved Riemannian structures.
Bull. Am. Math. Soc. {\bf 81}, 93--97 (1975)
https://doi.org/10.1090/S0002-9904-1975-13649-4

\bibitem{Bat2}
Batkhin, A.\,B.: A real variety with boundary
and its global parameterization. Program. Comput. Softw. {\bf 43}(2), 75--83 (2017)
https://doi.org/10.1134/S0361768817020037


\bibitem{Bat}
Batkhin, A.\,B., Bruno, A.\,D.:  Investigation of a real
algebraic surface. Program. Comput. Softw. {\bf 41}(2), 74--83. (2015)
https://doi.org/10.1134/S0361768815020036



\bibitem{BB}
 B\'erard Bergery, L.:
Les vari\'et\'es riemanniennes homog\`enes simplement connexes de dimension impaire
\`a courbure strictement positive.  J. Math. Pures Appl. {\bf 55}(1), 47--67 (1976)



\bibitem{Berest}
Berestovski\u\i\, V.\,N.:
Homogeneous Riemannian manifolds of positive Ricci curvature.
Mat. Zametki. {\bf 58}(3), 334--340, 478 (1995) (Russian).
[English translation. In: Math. Notes {\bf 58}(3-4), 905--909 (1995)]
https://doi.org/10.1007/BF02304766

\bibitem{BerNik2020}
Berestovski\u\i\, V.\,N,  Nikonorov, Yu.,\, N.: Riemannian Manifolds and Homogeneous Geodesics. Springer Monographs in Mathematics. Springer, Cham, 2020, XXII+482 pp.


\bibitem{Be}
Berger, M.: Les vari\'et\'es riemanniennes homog\`enes normales simplement connexes \`a courbure strictment positive.
Ann. Scuola Norm. Sup. Pisa. {\bf 15}(3), 179--246 (1961)
https://eudml.org/doc/83265


\bibitem{Bettiol}
Bettiol, R.\,G., Krishnan, A.\,M.:
Ricci flow does not preserve positive sectional curvature in dimension four.
Calc. Var. Partial Diff. Eq. {\bf 62}(1), 13, 21p. (2023)
https://doi.org/10.1007/s00526-022-02335-z

\bibitem{Bo}
B\"ohm, C., Wilking, B.:
Nonnegatively curved manifolds with finite fundamental groups admit metrics with positive Ricci curvature. GAFA Geom. Func. Anal. {\bf 17}, 665--681 (2007)
https://doi.org/10.1007/s00039-007-0617-8


\bibitem{Bruno}
Bruno, A.\,D., Azimov, A.\,A.:
Parametric expansions of an algebraic variety near
its  singularities. Axioms.  {\bf 12}(5), 469 (2023) https://doi.org/10.3390/axioms12050469

\bibitem{Bruno2}
Bruno, A.\,D., Azimov, A.\,A.:
Parametric expansions of an algebraic variety near
its  singularities II.  Axioms. {\bf 13}(2), 106 (2024) https://doi.org/10.3390/axioms13020106


\bibitem{Caven}
Cavenaghi, L.\,F., Grama, L., Martins, R.\,M.:
On the dynamics of positively curved metrics on $\rm SU(3)/T^2$ under the homogeneous Ricci flow.
Mat. Contemp. {\bf 60}, 3--30 (2024)
http://doi.org/10.21711/231766362024/rmc602

\bibitem{CKL}
 Chen, Z., Kang, Y., Liang, K.:
 Invariant Einstein metrics on three-locally-symmetric spaces.
 Commun. Anal. Geom. {\bf 24}(4), 769--792 (2016)
https://doi.org/10.4310/CAG.2016.v24.n4.a4

\bibitem{ChWal}
Cheung, M.-W., Wallach, N.\,R.:
Ricci flow and curvature on the variety of flags on the two dimensional projective space over the complexes, quaternions and octonions.
Proc. Am. Math. Soc.  {\bf 143}(1), 369--378 (2015)
https://doi.org/10.1090/S0002-9939-2014-12241-6

\bibitem{Gonzales} Gonz\'alez-\'Alvaro, D.,  Zarei, M.:
Positive intermediate curvatures and Ricci flow.
Proc. Am. Math. Soc. {\bf 152}(6), 2637--2645 (2024)
https://doi.org/10.1090/proc/16752

\bibitem{Ham}
 Hamilton, R.\,S.: Three-manifolds with positive Ricci curvature.
J. Diff. Geom. {\bf 17}, 255--306 (1982)
https://doi.org/10.4310/jdg/1214436922

\bibitem{Lomshakov2}
Lomshakov, A.\,M., Nikonorov, Yu.\,G., Firsov, E.\,V.:
Invariant Einstein metrics on three-locally-symmetric spaces.
Matem. Tr., {\bf 6}(2), 80--101  (2003) (Russian).
[English translation. In: Sib. Adv. Math., {\bf 14}(3), 43--62 (2004)]



\bibitem{Nikonorov2} Nikonorov, Yu.\,G.:
On a class of homogeneous compact
 Einstein manifolds. Sibirsk. Mat. Zh. {\bf 41}(1), 200--205 (2000)(Russian).
[English translation. In: Sib. Math. J. {\bf 41}(1), 168--172 (2000)]
https://doi.org/10.1007/BF02674006

\bibitem{Nikonorov4}
Nikonorov, Yu.\,G.: Classification of generalized Wallach spaces.
Geom. Dedicata. {\bf 181}(1), 193--212 (2016); correction: Geom. Dedicata.
{\bf 214}(1), 849-–851 (2021)
https://doi.org/10.1007/s10711-015-0119-z,
https://doi.org/10.1007/s10711-021-00604-3

\bibitem{Nikonorov1}
Nikonorov, Yu.\,G., Rodionov, E.\,D., Slavskii, V.\,V.:
Geometry of homogeneous Riemannian manifolds. J. Math. Sci. (New York)
{\bf 146}(7), 6313--6390 (2007)
https://doi.org/10.1007/s10958-007-0472-z

\bibitem{Stat} Statha, M.:  Ricci flow on certain homogeneous spaces.
Ann. Glob. Anal. Geom. {\bf 62}(1), 93--127 (2022)
https://doi.org/10.1007/s10455-022-09843-3


\bibitem{Wal}
Wallach, N.\,R.: Compact homogeneous Riemannian manifolds with
strictly positive curvature. Ann. Math. Second Ser. {\bf 96}(2), 277--295 (1972)
https://doi.org/10.2307/1970789

\bibitem{Wil} Wilking, B.:
The normal  homogeneous space $SU(3)\times SO(3)/U(2)$ has  positive sectional curvature.
Proc. Am. Math. Soc. {\bf 127}(4), 1191--1194 (1999)
https://doi.org/10.1090/S0002-9939-99-04613-4

\bibitem{WiZi}
Wilking, B., Ziller, W.:
Revisiting homogeneous spaces with positive curvature.
J. Reine Angew. Math. {\bf 738}, 313--328 (2018)
https://doi.org/10.1515/crelle-2015-0053

\bibitem{XuWolf}
Xu, M., Wolf, J.\,A.:
$Sp(2)/U(1)$ and a positive curvature problem.
Diff. Geom. Appl. {\bf 42}, 115--124 (2015)
https://doi.org/10.1016/j.difgeo.2015.08.002



\end{thebibliography}

\vspace{5mm}

\end{document}